\documentclass{article}

\usepackage{amsmath, amssymb, amsthm}
\usepackage{graphicx} % Required for inserting images
\usepackage{nickmath}
\usepackage[round]{natbib}
\usepackage{xcolor}
\usepackage[]{todonotes}
\usepackage{appendix}
\usepackage{setspace}
\usepackage{tikz}

 \usepackage{booktabs}
 \usepackage{subcaption}

\usetikzlibrary{calc}

\newtheorem{lem}{Lemma}
\newtheorem{cor}{Corollary}
\newtheorem{prp}{Proposition}
\newtheorem{thm}{Theorem}

\theoremstyle{definition}
\newtheorem{con}{Condition}
\newtheorem*{con*}{Condition}
\newtheorem{dfn}{Definition}
\newtheorem{axm}{Axiom}
\newtheorem*{axm*}{Axiom}

\theoremstyle{remark}
\newtheorem{rmk}{Remark}
\newtheorem{exm}{Example}

\title{The E-measure}
\author{Nick W. Koning\footnote{n.w.koning@ese.eur.nl}\\ Econometric Institute, Erasmus University Rotterdam}
\date{22 April 2026}

\begin{document}
    \maketitle

    \begin{abstract}
        We introduce the E-measure: a measure-like generalization of the E-value to a class of hypotheses.
        Unlike classical measures, E-measures are closed under infimums instead of addition.
        They arise from a compatibility axiom with logical implications, that there should be at least as much evidence against more specific hypotheses.
        We show that E-measures are the only non-dominated such objects, if the hypothesis class is closed under intersections.
        We propose to use the E-measure to present all the relevant evidence for a problem, where the relevance is captured by the choice of hypothesis class.
        We showcase this by applying the E-measure to decision making, inducing a hypothesis class from the uncertain consequences of decisions.
        This results in uniform E-consequence bounds on decisions, which nest high-probability loss bounds.
        Correcting for multiplicity, we consider `familywise evidence' and `false evidence rate' control, generalizing from errors and discoveries to continuous evidence.
        Remarkably, E-measures control these without multiplicity correction if the hypothesis class is intersection-closed.
        Moreover, we obtain a `frequentist' notion of updating from E-prior to E-posterior.
        Abstracting the notion of a `hypothesis', we advocate for using E-measures for any unknown quantity, leading to predictive E-measures.
        % Overall, we believe E-measures form the natural mathematical language of evidence.
    \end{abstract}
    
    \section{Introduction}
        Hypothesis testing is one of the cornerstones of statistics.
        The E-value generalizes a hypothesis test from a binary reject-or-not outcome to a continuous notion of evidence \citep{koning2024continuous}.
        Its inception has sparked a revolution in the way we look at uncertainty quantification \citep{shafer2021testing, vovk2021evalues, grunwald2024safe, ramdas2025hypothesis}.

        Many applications of the E-value involve putting together a collection of E-values for a collection of hypotheses.
        This includes developments in multiple testing \citep{wang2022false, xu2025bringing}, the construction of confidence sets \citep{howard2021time, waudbysmith2023estimating, koobs2026equivalence}, insights into sequential testing \citep{ramdas2022admissible, koning2026anytime}, and the study of loss bounds on decisions \citep{grunwald2023posterior, koning2025fuzzy}.

        The main contribution of this paper is to fuse a collection of E-values together into a single object: \emph{the E-measure}.
        The E-measure can be viewed as a measure-like generalization of the E-value from a single hypothesis to a class of hypotheses.
        It simultaneously presents the evidence against each hypothesis in the class.
        
        The E-measure is designed to satisfy a single axiom: the Closure axiom.
        This defining property of an E-measure can be viewed as a continuous generalization of the closure principle from binary multiple testing \citep{marcus1976closed}.
        
        \begin{axm*}[Closure]
            The evidence against a hypothesis equals the least evidence against the hypotheses that imply it.
        \end{axm*}          

        Our main result about the E-measure is to show that it can be forged from two natural more primitive axioms, which relate evidence to logical impossibilities and implications.
        In particular, we first extend the notion of validity of E-values to E-measures.
        We then show that any valid precursor to an E-measure satisfying both the Impossibility and Implication axioms is dominated by a valid E-measure, if the class of hypotheses is closed under intersections.
        This implies any valid and non-dominated such precursor must be an E-measure.
        
        \begin{axm*}[Impossibility]
            Logically impossible hypotheses receive maximal evidence.
        \end{axm*}
    
        \begin{axm*}[Implication]
            If a hypothesis logically implies another hypothesis, then there cannot be less evidence against it.
        \end{axm*}

        We propose to use the E-measure to capture all the relevant evidence for a problem, where `relevance' is described by the associated hypothesis class.
        We call this the \emph{E-principle} as a nod to the classical likelihood principle, which loosely states that all the relevant statistical information is captured by the likelihood.
        This E-principle can even be viewed as an extension of the likelihood principle, since (the reciprocal of) the likelihood corresponds to an E-measure for the power set hypothesis class; the most granular possible hypothesis class.
        
        Beyond likelihood functions, classical examples of (binary) E-measures are confidence sets and rejection sets of multiple testing procedures, because these can be interpreted as reporting a test outcome (binary E-value) for a collection of hypotheses.
        
        \subsection{Using E-measures}
            We develop multiple extensions and applications of E-measures, generalizing several state-of-the-art results.

            \begin{itemize}                    
                \item \textbf{Decision-making}.
                    We use E-measures for decision making under uncertainty.
                    Following the E-principle, we induce a hypothesis class that contains the relevant claims about the consequences of decisions.
                    We use this to develop E-consequence bounds, revealing a deep generalization of several recently studied probabilistic and E-value-based loss bounds \citep{grunwald2023posterior, andrews2025certified, kiyani2025decision, koning2025fuzzy}.

                \item \textbf{Multiplicity}.
                    We study cross-hypothesis notions of validity for E-measures to account for multiplicity.
                    This leads us to define notions of `Familywise Evidence' and `False Evidence Rate'.
                    These extend the classical Familywise Error Rate and False Discovery Rate from binary testing to continuous evidence.
                    If the hypothesis class is intersection-closed, we show that plain hypothesis-wise validity is equivalent to both Familywise Evidence Control, and uniform False Evidence Rate control.
                    We describe how these connect to recent E-value-based methods for the False Discovery Rate \citep{wang2022false, xu2025bringing}.
                    Going beyond the Familywise Evidence and False Evidence Rate, we obtain continuous analogues of the general binary multiple testing errors considered by \citet{xu2025bringing}.
                
                \item \textbf{From E-prior to E-posterior}. 
                    We develop a `frequentist' notion of updating, going from a prior E-measure to a posterior E-measure.
                    To ensure that the posterior is an E-measure, we rely on the intersection-closure assumption and we define a notion of closed-multiplication of E-measures.

                \item \textbf{E-measure processes}. 
                    As E-values are highly popular in sequential settings, we develop an E-measure generalization: the E-measure process.
                    If the hypothesis class is intersection-closed, we find that a non-dominated E-measure process evaluated at a hypothesis equals the infimum over the E-measure processes for the least hypotheses that imply it.
                    This recovers the landmark result of \citet{ramdas2022admissible}, that an admissible E-process for a composite hypothesis must equal the infimum of pointwise E-processes.

                \item \textbf{Abstracting hypotheses}.
                    The standard convention in statistics is that a `hypothesis' is a collection of probability distributions.
                    We find that our framework does not rely on this convention.
                    Breaking the convention, we recover the foundation of predictive inference by taking hypotheses as subsets of a sample space.
                    We show this recovers a result of \citet{koning2025fuzzy} that a single valid E-value suffices for predictive validity.
                    We also show how one can pushforward E-measures and their validity onto different spaces.
                    
                \item \textbf{E-integration}. 
                    In Appendix \ref{sec:E-integration}, we study integration of non-negative functions under E-measures, deriving E-Markov and Post-hoc E-Markov (in)equalities.
                    If the hypothesis class is intersection-closed, this corresponds to a notion of loss-aggregation developed by \citet{grunwald2023posterior}.
            \end{itemize}

        \subsection{Relationship to literature}
            This paper builds upon the pioneering work of \citet{grunwald2023posterior}.
            He coins the term `E-posterior' for a collection of E-values as a function over a parameter space, applying it to derive data-informed loss bounds for decision making.
            We dedicate Section \ref{sec:grunwald} to explain the relationship to our results on decision making.
            There, we find that the loss bound of Gr\"unwald follows from applying the E-Markov inequality to our E-consequence bound.
            We reserve the term `E-posterior' for settings in which we have an explicit `E-prior'.
            
            Our notion of closure can be viewed as a continuous-evidence analogue of the classical closure principle for collections of binary `reject-or-not' tests in multiple testing \citep{marcus1976closed}.
            Thresholding an E-measure at a fixed level $\alpha$ recovers a closed testing procedure.
            In such a binary setting, the Implication axiom translates into the statement that if a rejected hypothesis implies another hypothesis, then the implied hypothesis must also be rejected.
            The closure axiom translates into the closure of a rejection set.

            Recent work by \citet{xu2025bringing} develops an `e-Closure Principle' for binary multiple testing, showing that multiple testing procedures controlling an expected testing loss can be represented through specific collections of E-values.
            A deep link to our work is that they find that the validity of such multiple testing procedures hangs on the validity of E-values for the least (`intersection') hypotheses.
            Such least hypotheses also prominently feature in our work.
            We detail the connection to their main result in Section \ref{sec:general_multiplicity}.
            
            The (reciprocal of) an E-measure has been abstractly studied under the name `maxitive measures' introduced by \citet{shilkret1971maxitive}.
            From this angle, our work can be viewed as uniting maxitive measures with E-values, hypotheses and statistical validity, as well as the closure principle from multiple testing.
            
    \section{Hypothesis classes and spaces}\label{sec:hypothesis_class}
        We consider a measurable space $(\X, \Sigma)$, where $\X$ represents the sample space and $\Sigma$ is a $\sigma$-algebra of subsets of $\X$ that represents the available information.
        We equip this space with a model $\P$: some collection of probability measures on $\X$.
        We follow the standard convention in statistics by defining a hypothesis $H$ as a subset of this model $H \subseteq \P$: some collection of probabilities.\footnote{We propose to break this convention in Section \ref{sec:abstract_hypotheses}.}

        To facilitate the interpretation, it will often be helpful to think of a `true probability' $P^* \in \P$ that we intend to recover.
        A hypothesis $H$ may then be interpreted as the statement that it contains the true probability: `$P^* \in H$'.
        We say that a hypothesis is \emph{true} if $P^* \in H$.
        We say that a hypothesis $H$ \emph{implies} a hypothesis $H'$ if $H \subseteq H'$, since $P^* \in H$ then implies $P^* \in H'$.
        
        A central object in this work is a hypothesis class $\H$, which we define as some union-closed set of hypotheses.
        The hypothesis class $\H$ can be interpreted as representing the statements about $P^*$ that we wish to learn about.
        We equip $\P$ with a hypothesis class to create a measurable-like \emph{hypothesis space} $(\P, \H)$.
        
        \begin{dfn}[Hypothesis class]
            A collection $\mathcal{H}$ of hypotheses $H \subseteq \mathcal{P}$ is a hypothesis class if it is closed under (arbitrary) unions.
        \end{dfn}

        \begin{dfn}[Hypothesis space]
            We say that $(\P, \H)$ is a hypothesis space if $\H$ is a hypothesis class of subsets of $\P$.
        \end{dfn}
        
        \begin{rmk}[Empty set]
            We consider the empty union to be among the unions, so that a hypothesis class contains the empty set $\emptyset$.
            We write $\H^\emptyset$ as a shorthand for $\H \setminus \{\emptyset\}$.
        \end{rmk}
        
        We say that a hypothesis $H$ is $\mathcal{H}$-\emph{measurable} if $H \in \mathcal{H}$, so that measurable statements are exactly those statements about $P^*$ we can consider with the hypothesis class.
        A function $f : \P \to \P'$ between two hypothesis spaces $(\P, \H)$, $(\P', \H')$ is said to be ($\H$, $\H'$)-measurable if $f^{-1}(H') \in \H$, for every $H' \in \H'$.
        We use the shorthand $\H$-measurable for a function $f : \P \to \P'$ that is $(\H, 2^{\P'})$-measurable.
        We say that a subset $\S \subseteq \H$ \emph{generates} $\H$ if the union closure of $\S$ is $\H$.
        
        \begin{exm}[Single hypothesis]
            A smallest non-trivial hypothesis class is $\H = \{\emptyset, H\}$ for some hypothesis $H$.
        \end{exm}
        
        \begin{exm}[Nested]\label{exm:nested}
            The union-closure of a nested collection of hypotheses is again nested, and so a hypothesis class.
            Such a nested hypothesis class underlies `one-sided' or `non-inferiority' testing, as recently studied in an E-value context in \citet{koobs2026equivalence}.
        \end{exm}

        \subsection{Least hypotheses and intersection-closure}
            While closure of a hypothesis class $\H$ under arbitrary unions suffices for defining an E-measure, many of our results rely on the additional condition that $\H$ is also closed under (arbitrary) intersections.
            
            \begin{con}[Intersection-closure]
                We say a hypothesis space $(\P, \H)$ is \emph{intersection-closed} if $\H$ is closed under (arbitrary) intersections and $\P \in \H$.
            \end{con}
            
            \begin{rmk}[Not a $\sigma$-algebra]
                We stress that intersection-closure does not make $\H$ a $\sigma$-algebra, as we do not assume closure under complements.
                As a consequence, a hypothesis $H$ does not necessarily come with a natural `alternative hypothesis' $H^C$ in this framework.
            \end{rmk}

            \begin{rmk}[Alexandrov topology]
                A space that is closed under arbitrary unions and intersections, and contains $\emptyset$ and $\P$ is also known as an Alexandrov topology.
            \end{rmk}

            Lemma \ref{lem:intersection_closed} relates intersection-closure to the existence of a \emph{least hypothesis} $H_P$ for every point $P \in \P$: a `smallest' hypothesis that contains $P$.
            Its proof, alongside all other omitted proofs can be found in the appendix.
            
            \begin{dfn}[Least hypothesis]
                Given some $P \in \P$, we say that $H_P$ is its least hypothesis if $P \in H_P$ and $P \in H$ implies that $H \supseteq H_P$, for every $H \in \H$.
            \end{dfn}

            \begin{lem}[Least hypothesis and intersection-closure]\label{lem:intersection_closed}
                Every $P \in \P$ has a least hypothesis $H_P \in \H$ if and only if $(\P, \H)$ is intersection-closed.
                Under these equivalent conditions, the least hypothesis of $P \in \P$ is
                \begin{align*}
                    H_P = \bigcap\{H \in \H : P \in H\}.
                \end{align*}
            \end{lem}

            In Lemma \ref{lem:canonical_cover}, we further show that intersection-closure is equivalent to the existence of a canonical cover for every hypothesis $H \in \H$: the least hypotheses $\{H_P : P \in H\}$.
            An implication is that the collection of least hypotheses generates $\H$.
            
            \begin{lem}[Canonical cover]\label{lem:canonical_cover}
                Suppose that $H_P \in \H$ for every $P \in \P$.
                Then $(\P, \H)$ is intersection-closed with least hypotheses $H_P$, $P \in \P$, if and only if both $P \in H_P$ for every $P \in \P$ and
                \begin{align*}
                    H = \bigcup_{P \in H} H_P, \quad \textnormal{for every } H \in \H.
                \end{align*}
            \end{lem}

            \begin{exm}[Power set]
                If $\H = 2^\P$, then the least hypotheses are the singleton sets $\{P\}$, $P \in \P$.
                Here, the canonical cover of $H$ is simply the collection of its singletons.
            \end{exm}

            \begin{exm}[Partition]
                Suppose that $\P$ is partitioned into disjoint cells, and that $\H$ is the union-closure of these cells.
                Then every $P \in \P$ has a least hypothesis $H_P$: the cell that contains $P$.
                Here, different points may share the same least hypothesis $H_P$, reflecting that $\H$ does not distinguish between them.
            \end{exm}

            \begin{exm}[Overlap]
                Let $\H = \{\emptyset, \{P_1\}, \{P_1, P_2\}, \{P_1, P_3\}, \{P_1, P_2, P_3\}\}$, with $\P = \{P_1, P_2, P_3\}$.
                Then, $(\P, \H)$ is intersection-closed, and the least hypotheses are $H_{P_1} = \{P_1\}$, $H_{P_2} = \{P_1, P_2\}$ and $H_{P_3} = \{P_1, P_3\}$.
                This shows the least hypotheses need not be cells, as they may overlap.
            \end{exm}

            \begin{exm}[Nested]
                In the nested hypothesis class from Example \ref{exm:nested}, a hypothesis $H$ is the least hypothesis for $P$ if and only if $P \in H \setminus \cup\{H' \in \H : H' \subsetneq H\}$.
            \end{exm}

        \subsection{Intersection-closure and preorders}\label{sec:preorder_induced}
            The intersection-closure property is an order-theoretic structure on $\P$.
            Indeed, intersection-closed hypothesis spaces are exactly those induced by a preorder $\precsim$ on $\P$.
            This characterization is the key to our results in Section \ref{sec:decisions}, where we induce a hypothesis class from a natural preorder on the consequences of decisions.
            
            \begin{dfn}[Preorder-induced hypothesis class]
                Let $\precsim$ be a preorder on $\P$.
                For every $P \in \P$, we define the principal upper set
                \begin{align*}
                    H_\precsim(P)
                        := \{P' \in \P : P \precsim P'\}, \quad P \in \P.
                \end{align*}
                We define $\H_\precsim$ as the smallest (union-closed) hypothesis class containing all these principal upper sets $H_\precsim(P)$, $P \in \P$.
            \end{dfn}

            \begin{prp}\label{prp:order_least_true}
                A hypothesis space $(\P, \H)$ is intersection-closed if and only if there exists a preorder $\precsim$ on $\P$ such that $\H = \H_\precsim$.
                Under these equivalent conditions, the preorder is determined by $\H$ through $P \precsim_\H P' \iff P' \in H_P$, and the principal upper sets $H_\precsim(P)$ are the least hypotheses.
            \end{prp}
            
    \section{Evidence axioms and E-measures}
        \subsection{E-values}
            We quantify the evidence (`E-value') against a hypothesis $H \in \H$ as a value in $[0, \infty]$.
            Here, we use the convention that larger values correspond to stronger evidence.
            Throughout, we latently assume that we prefer to have more evidence.

            In practice, an E-value will be informed by data, in which case we refer to it as an E-variable.
            For the sake of presentation, we postpone the introduction of data to later sections.
            
        \subsection{Impossibility axiom and E-functions}
            Our goal is to present evidence not against a single hypothesis $H$, but against an entire hypothesis class $\H$.
            For this purpose, we introduce \emph{E-functions}.
            
            For a given hypothesis $H$, an E-function returns an amount of evidence (E-value) in $[0, \infty]$ against this hypothesis.
            That is, evidence against the statement that $H$ is true: $P^* \in H$.
            
            \begin{dfn}[E-function]
                A map $\e : \H \to [0, \infty]$ is an E-function for a hypothesis class $\H$ if $\e(\emptyset) = \infty$.
            \end{dfn}

            In our definition of an E-function, we integrate the \emph{Impossibility axiom}, which we believe any reasonable notion of evidence should satisfy: logical impossibilities correspond to maximum evidence.
            As $\H$ contains the empty set $\emptyset$, this immediately translates to the (single) defining property of an E-function, expressing that the statement $P^* \in \emptyset$ is logically impossible.
            
            \begin{axm}[Impossibility]
                We have maximal evidence against logically impossible hypotheses.
            \end{axm}

            Example \ref{exm:test_outcomes} shows that E-functions can be viewed as a continuous generalization both of discovery sets from multiple testing and of confidence sets.
            
            \begin{exm}[Test outcomes as binary E-functions]\label{exm:test_outcomes}
                An example of an E-function is a function (collection) of level $\alpha$ test outcomes: $\e_\alpha : \H^\emptyset \to \{0, 1/\alpha\}$, and $\e_\alpha(\emptyset) = \infty$.
                For a given hypothesis $H \in \H^\emptyset$, such a test outcome is either $\e_\alpha(H) = 0$, encoding a non-rejection, or $\e_\alpha(H) = 1/\alpha$, encoding a rejection at level $\alpha$.
                While these values are somewhat arbitrary now, they are convenient to recover classical Type-I error validity later.
                
                A collection of test outcomes for hypotheses is the typical output of a multiple testing procedure, where the rejected hypotheses are often collected into a discovery set $R_\alpha = \{H \in \H^\emptyset : \e(H) \geq 1/\alpha\}$.

                A common complementary description of a collection of test outcomes is as a confidence set, which instead groups the hypotheses that are not rejected: $C_\alpha = \{H : \e_\alpha(H) < 1 / \alpha\} = R_\alpha^C \setminus \{\emptyset\}$.
                The test outcomes are recovered through test inversion: $\e_\alpha(H) = 1 / \alpha \times \ind{H \not\in C_\alpha}$, $H \in \H^\emptyset$.
            \end{exm}

            \begin{rmk}[Infinity conventions]
                Throughout, we use the conventions $\inf \emptyset = \infty$, $\sup \emptyset = 0$, $c / 0 = \infty$, $c / \infty = 0$, $c > 0$, $0 / 0 = 0$ and $0 \times \infty = \infty \times 0 = 0$.
            \end{rmk}

        \subsection{Implication axiom and E-capacities}
            A second axiom that we believe any reasonable notion of evidence should satisfy is the \emph{Implication axiom}.
            This axiom expresses coherence with logical implications: the hypothesis $H = \{P_1\}$ that `$P_1$ is true' is easier to refute than the hypothesis $H' = \{P_1, P_2\}$ that `either $P_1$ or $P_2$ is true'.
            
            \begin{axm}[Implication]
                If a hypothesis $H$ logically implies another hypothesis $H'$ then we should have at least as much evidence against $H$ as we have against $H'$.
            \end{axm}
            
            We translate this axiom into the defining property of an \emph{E-capacity}.
            This corresponds to assuming that an E-function $\e$ is \emph{antitonic} on $\H$ under set-inclusion.
            
            \begin{dfn}[E-capacity]
                An E-function $\e$ is an E-capacity for a hypothesis class $\H$ if $H \subseteq H'$ implies $\e(H') \leq \e(H)$, for every $H, H' \in \mathcal{H}$.
            \end{dfn}

        \subsection{Closure axiom and E-measures}
            A stronger axiom is the \emph{Closure axiom}, which implies both the implication and impossibility axioms.
            Translating this axiom into a defining property, we obtain the \emph{E-measure}.
            An E-measure is also an E-capacity, so that an E-measure can be viewed as an E-capacity with additional structure.
            
            \begin{axm}[Closure]
                The evidence against a hypothesis is the smallest evidence against any of its sub-hypotheses.
            \end{axm}

            \begin{dfn}[E-measure]
                We say that $\e : \H \to [0, \infty]$ is an E-measure for $\H$ if, for every subset $\S \subseteq \H$,
                \begin{align*}
                    \e\left(\bigcup_{H \in \S} H\right)
                        = \inf_{H \in \S} \e(H)
                \end{align*}
            \end{dfn}

            \begin{lem}
                An E-measure is an E-function.
            \end{lem}
            \begin{proof}
                We consider the empty set to be among the subsets $\mathcal{S} \subseteq \mathcal{H}$, so that $\inf \emptyset = \infty$ leads to $\e(\emptyset) = \infty$.
            \end{proof}

            \begin{lem}
                An E-measure is an E-capacity.
            \end{lem}
            \begin{proof}
                $H \subseteq H'$ implies $\e(H') = \e(H \cup H') = \min\{\e(H), \e(H')\} \leq \e(H)$.
            \end{proof}

            \begin{rmk}[Closed testing]
                If $\e$ is an E-measure, then the rejection set $R_\alpha = \{H \in \H^\emptyset : \e(H) \geq 1/\alpha\}$ is closed under unions: $\bigcup_{H \in \S} H \in R_\alpha \iff \S \subseteq R_\alpha$, for every subset $\S \subseteq \H^\emptyset$.
                Indeed, by the Closure axiom,
                \begin{align*}
                    \e\left(\bigcup_{H \in \S} H\right) \geq 1/\alpha
                        &\iff \inf_{H \in \S} \e\left(H\right) \geq 1/\alpha \\
                        &\iff \e\left(H\right) \geq 1/\alpha, \quad \textnormal{ for every } H \in \S.
                \end{align*}
                This recovers the closed testing principle: a union of hypotheses is rejected if and only if every hypothesis in the union is rejected.
                
                If we restrict ourselves to $\{0, 1/\alpha\}$-valued E-measures, then E-measures are equivalent to closed rejection sets.
            \end{rmk}

            \begin{rmk}[Measures, p-values, maxitive measure and capacities]
                E-measures are inverted compared to traditional measures, assigning smaller values to larger sets.
                If monotonicity is preferred over antitonicity, we can simply flip from the e-value scale to the p-value scale: $p = 1 / \e$.
                The `$p$-measure' $p : \mathcal{H} \to [0, \infty]$ satisfies $p\left(\bigcup_{A \in \mathcal{S}} A\right) = \sup_{A \in \mathcal{S}} p(A)$, swapping minimization for maximization.
                The closure axiom can be interpreted as replacing the closure axiom of classical measures under addition by closure under infimums.
                On this flipped scale, an E-capacity is known simply as a capacity, and an E-measure as a maxitive measure \citep{shilkret1971maxitive}.
            \end{rmk}

        \subsection{E-densities}
            E-measures offer an important practical advantage over E-capacities: if we know the value of an E-measure for a subset of hypotheses $\mathcal{S} \subseteq \mathcal{H}$ that generates $\mathcal{H}$, then we can use the closure axiom to recover the full E-measure on $\mathcal{H}$.

            If $(\P, \H)$ is intersection-closed, then the collection of least hypotheses $\{H_P : P \in \P\}$ generates $\H$.
            This means it suffices to present the evidence over this class of least hypotheses $H_P \mapsto \e(H_P)$.
            We refer to this as an \emph{E-density}.
            This is not true for E-capacities, as their defining property only provides bounds on the evidence of other hypotheses, which is not enough to recover the entire E-capacity.

            \begin{rmk}
                If desired, one could extend this E-density over the least hypotheses to $2^\P$ by considering $P \mapsto \e(H_P)$.
                This of course does not provide any additional information.
                We explore such extensions to $2^\P$ in more detail in Appendix \ref{appn:canonical_extension}.
            \end{rmk}
            
        \section{Closing E-functions into E-measures}\label{sec:closing_e_functions}
            Every E-function $\e$ can be \emph{closed} into an E-measure $\overline{\e}$.
            This \emph{closure} of $\e$ is the smallest E-measure that dominates it.

            For a given hypothesis $H$, the closure can be viewed as correcting $\e(H)$ upward as little as possible.
            For each cover $\T$ of $H$, the closure axiom forces the value at $H$ to be at least $\inf_{H' \in \T} \e(H')$.
            Taking the supremum over all covers yields the strongest such lower bound: $\overline{\e}(H)$.
            
            \begin{dfn}[Closure]\label{dfn:closure}
                Given an E-function $\e$, we define its closure $\overline{\e}$ as
                \begin{align}\label{eq:closure}
                     \overline{\e}(H)
                            = \sup_{\T \in \C(H)} \inf_{H' \in \T} \e(H'),
                \end{align}
                where $\C(H) := \{\T \subseteq \H : H \subseteq \cup_{H' \in \T} H'\}$ denotes the collection of covers of the hypothesis $H$.
            \end{dfn}
            
            \begin{prp}\label{prp:induced_E-measure}
                If $\e$ is an E-function then its closure $\overline{\e}$ is an E-measure.
            \end{prp}
            
            \begin{thm}[Smallest dominating E-measure]\label{thm:smallest_dominating_E-measure}
                An E-function $\e$ is dominated by its closure $\overline{\e}$.
                Every E-measure $\e'$ that dominates an E-function $\e$ also dominates its closure $\overline{\e}$.
                In particular, if $\e$ is an E-measure, then $\e = \overline{\e}$.
            \end{thm}
            
            \begin{dfn}[Domination]
                We say that an E-function $\e$ dominates another E-function $\e'$ if $\e(H) \geq \e'(H)$, for every $H \in \H$.
            \end{dfn}
  
            \begin{rmk}[Closure operator]
                Theorem \ref{thm:smallest_dominating_E-measure} shows that $\e \mapsto \overline{\e}$ is a `closure operator' on the set of E-functions.
            \end{rmk}
            
            \begin{exm}
                Suppose $\P = \{P_1, P_2\}$ and $\H = 2^\P$, which has least hypotheses $\{P_1\}$ and $\{P_2\}$.
                Consider the E-capacity $\e(\emptyset) = \infty$, $\e(\{P_1\}) = 4$, $\e(\{P_2\}) = 2$, $\e(\{P_1, P_2\}) = 1$.
                This is not an E-measure, as the closure axiom would require $\e(\{P_1, P_2\}) = \inf\{\e(\{P_1\}), \e(\{P_2\})\} = \inf\{4, 2\} = 2$.
                Its closure $\overline{\e}$ coincides with $\e$, except for $\overline{\e}(\{P_1, P_2\}) = 2 > 1 = \e(\{P_1, P_2\})$.
            \end{exm}
            
        \subsection{Closure for intersection-closed hypothesis classes}
            The construction \eqref{eq:closure} of the closure $\overline{\e}$ of an E-function $\e$ maximizes over all covers.
            Theorem \ref{thm:closure_least_true_hypothesis} shows that if $(\P, \H)$ is intersection-closed, the canonical cover $H = \bigcup_{P \in H} H_P$ attains the optimum if $\e$ is an E-capacity.
            This result is fundamental to many of our later results.

            \begin{thm}\label{thm:closure_least_true_hypothesis}
                Suppose $(\P, \H)$ is intersection-closed with least hypotheses $H_P$, $P \in \P$.
                Let $\e$ be an E-capacity.
                Then, its closure equals
                \begin{align*}
                    \overline{\e}(H)
                        = \inf_{P \in H} \e(H_P), \quad H \in \H.
                \end{align*}
            \end{thm}

            Under this condition, we obtain a nice interpretation of the closure $\overline{\e}$ as the unique E-measure on $\H$ that agrees with the E-capacity $\e$ on the least hypotheses.

            \begin{cor}\label{cor:closure_least_true}
                Under the conditions of Theorem \ref{thm:closure_least_true_hypothesis}, $\overline{\e}$ is the unique E-measure for which $\overline{\e}(H_P) = \e(H_P)$, for every least hypothesis $H_P$, $P \in \P$.
            \end{cor}
        
    \section{E-kernels}
        We now move to informing our E-measures by data.
        For this purpose, we return to our sample space $\X$ equipped with a $\sigma$-algebra $\Sigma$.

        \subsection{E-variables}
            Most attention in the recent e-value literature has gone to measuring evidence against a single hypothesis using an \emph{E-variable}.
    
            \begin{dfn}[E-variable]
                A $\Sigma$-measurable map $\e(H \mid \cdot) : \mathcal{X} \to [0, \infty]$ is called an E-variable for $H$.
            \end{dfn}
    
            \begin{dfn}[Valid E-variable]
                An E-variable $x \mapsto \e(H \mid x) $ is valid if $\Ex^P[\e(H \mid X)] \leq 1$, for every $P \in H$.
            \end{dfn}
    
            \begin{rmk}[Separating E-variable and validity]
                We follow the convention from testing here, where a test is defined independently from its validity.
                This goes against the convention in the E-value literature, where the word `E-variable' is often reserved for what we call a `valid E-variable'.
            \end{rmk}

        \subsection{E-kernels}
            We introduce the E-kernel, which generalizes the E-variable to a hypothesis class $\H$.
            This can be interpreted as the E-analogue of a kernel for measures.
            
            \begin{dfn}[E-kernel]
                $\e : \H \times \X \to [0, \infty]$ is an E-kernel if both
                \begin{itemize}
                    \item The map $x \mapsto \e(H \mid x)$ is $\Sigma$-measurable for every $H \in \H$,
                    \item The map $H \mapsto \e(H \mid x)$ is an E-measure for every $x \in \mathcal{X}$.
                \end{itemize}
            \end{dfn}

            \begin{dfn}[Hypothesis-wise validity]
                We say that an E-kernel is (hypothesis-wise) valid if $x \mapsto \e(H \mid x)$ is a valid E-variable, for every $H \in \H^\emptyset$.
            \end{dfn}

            \begin{rmk}
                We can of course analogously define E-function kernels and E-capacity kernels, and refer to E-kernels as E-measure kernels.
                We distinguish between these when important, but use `E-kernel' as an umbrella term.
            \end{rmk}
            
            \begin{rmk}[\citet{grunwald2023posterior}]
                The `E-posterior' of \citet{grunwald2023posterior} corresponds to a hypothesis-wise valid E-function kernel (density) for $\H = 2^\P$, on the p-value scale $p = 1/\e$.
                Our E-kernel can be viewed as a generalization of this concept to arbitrary hypothesis classes.
                Moreover, we prefer to reserve the name `E-posterior' for an E-kernel relative to a prior E-measure; see Section \ref{sec:E-posterior}.
            \end{rmk}

        \subsection{Closure under hypothesis-wise validity}
            In Theorem \ref{thm:closure_pointwise_validity}, we present a first application of Theorem \ref{thm:closure_least_true_hypothesis} for (hypothesis-wise) valid E-kernels.
            It shows that if $(\P, \H)$ is intersection-closed then every non-dominated valid E-capacity kernel equals its closure and is therefore an E-measure kernel.
            Assuming we prefer more evidence, this implies that it suffices to consider E-measures among the class of E-capacities.
    
            \begin{thm}[Closure under validity]\label{thm:closure_pointwise_validity}
                Suppose that $(\P, \H)$ is intersection-closed, with an at most countable collection $\{H_P : P \in \P\}$ of least hypotheses.
                If $\e$ is an E-capacity kernel, then its closure $\overline{\e}$ is an E-measure kernel that dominates $\e$.
                $\e$ is valid if and only if $\overline{\e}$ is valid.
            \end{thm}
            \begin{proof}
                First, fix $x \in \X$.
                Since $H \mapsto \e(H \mid x)$ is an E-capacity, Proposition \ref{prp:induced_E-measure} implies that its closure is an E-measure.
                Moreover, Theorem \ref{thm:smallest_dominating_E-measure} implies that it dominates $\e$: $\overline{\e}(H \mid x) \geq \e(H \mid x)$, for every $H \in \H$.

                Now, fix $H \in \H^\emptyset$.
                By Theorem \ref{thm:closure_least_true_hypothesis}, $\overline{\e}(H \mid x) = \inf_{P \in H} \e(H_P \mid x)$.
                As $\{H_P : P \in \P\}$ is at most countable, this is an infimum over an at most countable family.
                As a consequence, $\Sigma$-measurability of $x \mapsto \e(H_P \mid x)$ implies that $x \mapsto \overline{\e}(H \mid x)$ is measurable.

                It remains to prove the validity claim.
                For every $P \in H$, we have $\overline{\e}(H \mid x) \leq \overline{\e}(H_P \mid x)$, since $\overline{\e}$ is an E-measure and so an E-capacity.
                By Corollary \ref{cor:closure_least_true}, $\overline{\e}(H_P \mid x) = \e(H_P \mid x)$.
                Hence, if $\e$ is valid then $\overline{\e}$ is valid:
                \begin{align*}
                    \Ex^{P}[\overline{\e}(H \mid X)]
                        \leq \Ex^{P}[\overline{\e}(H_P \mid X)]
                        = \Ex^{P}[\e(H_P \mid X)]
                        \leq 1,\quad \textnormal{ for every } P \in H.
                \end{align*}
                Conversely, if $\overline{\e}$ is valid then $\e$ is valid, since $\overline{\e}$ dominates $\e$.
            \end{proof}

            \begin{rmk}[Other notions of validity]\label{rmk:closure_other_validity}
                The core idea underlying the result is that closure leaves the evidence on the least hypotheses unchanged; see Corollary \ref{cor:closure_least_true}.
                As a consequence, the result extends to any notion of validity of an E-capacity $\e$ that is completely determined by the evidence $\e(H_P)$ against the least hypotheses $H_P$, $P \in \P$.
                Hypothesis-wise validity is one such notion: for an E-capacity kernel $\e$, validity for all least hypotheses $H_P$ implies validity for every $H \in \H^\emptyset$, since $\e(H) \leq \e(H_P)$ for every $P \in H$.
            \end{rmk}
    
            \begin{rmk}[At-most-countability]
                In Theorem \ref{thm:closure_pointwise_validity}, the at-most-countable assumption on the collection of least hypotheses is used only to ensure that the closure $x \mapsto \overline{\e}(H \mid x) = \inf_{P \in H} \e(H_P \mid x)$ is $\Sigma$-measurable.
                Indeed, for every $c > 0$, we have $\{x : \overline{\e}(H \mid x) \geq c\} = \bigcap_{P \in H} \{x : \e(H_P \mid x) \geq c\}$.
                Countability then suffices, because a $\sigma$-algebra is closed under countable intersections.
                We stress that at-most-countability of the collection of least hypotheses $\{H_P : P \in \P\}$ does not imply that $\H$ or $\P$ is at most countable.

                At-most countability can of course be replaced by the plain assumption that $\overline{\e}$ is measurable.
            \end{rmk}
            
        \subsection{Merging valid E-kernels}
            Valid E-capacity kernels inherit the desirable merging property of E-variables.

            \begin{prp}[Merging]\label{prp:merging}
                An at-most-countable convex combination of valid E-capacity kernels is a valid E-capacity kernel.
            \end{prp}
            
            \begin{rmk}[Merge then close]
                A weighted average of valid E-measure kernels is not guaranteed to be an E-measure kernel: it may only be an E-capacity kernel.
                Under the assumptions of Theorem \ref{thm:closure_pointwise_validity}, we can close the weighted average back into a valid E-measure kernel.
            \end{rmk}

        \subsection{Example: confidence sets}\label{sec:confidence_sets}
            A simple example of a valid E-kernel is a valid confidence set.
            Following Example \ref{exm:test_outcomes}, suppose that $x \mapsto \e(H \mid x)$ is $\{0, 1/\alpha\}$-valued (a level-$\alpha$ test), for every $H \in \H^\emptyset$.
            Collecting the hypotheses for which these tests do not reject produces a confidence set over $\H$:
            \begin{align*}
                C_\alpha(x) := \{H \in \mathcal{H}: \e(H \mid x) < 1/\alpha\},
            \end{align*}
            so that $\e(H \mid x) = 1/\alpha \ind{H \not\in C_\alpha(x)}$ for $H \in \H^\emptyset$.
            If the E-kernel is valid,
            \begin{align*}
                \Ex^P\left[\frac{1}{\alpha} \ind{H \not\in C_\alpha}\right] \leq 1, \quad \textnormal{ for every } P \in H, H \in \H^\emptyset.
            \end{align*}
            Rewriting this shows that $x \mapsto C_\alpha(x)$ is a valid level $\alpha$ confidence set:
            \begin{align*}
               P(H \in C_\alpha) \geq 1 - \alpha,\quad \textnormal{ for every } P \in H, H \in \H^\emptyset.
            \end{align*}

            \begin{rmk}[Confidence set for $H$]
                An interpretation of such a confidence set over hypotheses is that every true hypothesis $H \ni P^*$ in $\H$ is contained in $C_\alpha$, with high probability:
                \begin{align*}
                    P^*(H \in C_\alpha) \geq 1 - \alpha.
                \end{align*}
                In certain contexts, such a confidence set over hypotheses is called a `model confidence set'.
            \end{rmk}

            \begin{rmk}[Confidence set for $P^*$]\label{rmk:confidence_set_LTH}
                If $(\P, \H)$ is intersection-closed, we can construct confidence sets on $\P$:
                \begin{align*}
                    C_\alpha^\P(x)
                        := \{P \in \P : \e(H_P \mid x) < 1/\alpha\},
                \end{align*}
                which then satisfies $P(P \in C_\alpha^\P) \geq 1 - \alpha$, for every $P \in \P$.
                That is, the true probability $P^*$ is contained in $C_\alpha^\P$ with high probability.
                This is a special case of the extension of $\e$ from $\H$ to $2^\P$, discussed in Appendix \ref{appn:canonical_extension}.
            \end{rmk}

        \subsection{Equivalence to post-hoc confidence sets}\label{sec:post-hoc_confidence_sets}
            A deeper link between confidence sets and E-kernels is presented in Proposition \ref{prp:post-hoc}: a valid E-kernel is equivalent to a post-hoc valid collection of confidence sets.
            This relies on ideas on post-hoc $\alpha$ hypothesis testing developed by \citet{grunwald2024beyond} and \citet{koning2023post}.
            The result is effectively a reformulation of Theorem 2 in \citet{koning2023post} in terms of confidence sets.
            
            \begin{prp}[Post-hoc validity]\label{prp:post-hoc}
                $\e$ is a valid E-kernel if and only if for every data-dependent level $\widetilde{\alpha} \in [0, \infty]$,
                \begin{align}\label{ineq:post-hoc_valid_confidence_set}
                    \Ex^P\left[\frac{P(H \not\in C_{\widetilde{\alpha}} \mid \widetilde{\alpha})}{\widetilde{\alpha}}\right] \leq 1,\quad \textnormal{ for every } P \in H, H \in \H^\emptyset.
                \end{align}
            \end{prp}
            
            To relate this to classical confidence sets, we can rewrite classical validity as: for every \emph{data-independent} $\alpha \geq 0$,
            \begin{align}\label{ineq:valid_confidence_set}
                \frac{P(H \not\in C_\alpha)}{\alpha} \leq 1,\quad \textnormal{ for every } P \in H, H \in \H^\emptyset.
            \end{align}
            Comparing the two, \eqref{ineq:post-hoc_valid_confidence_set} can be interpreted as requiring \eqref{ineq:valid_confidence_set} to hold in expectation over every \emph{data-dependent} level.

            \begin{rmk}
                The relationship between confidence sets, E-function kernels (under the name `fuzzy prediction sets') and post-hoc confidence sets is central to \citet{koning2025fuzzy}, in the context of prediction sets; see Section \ref{sec:prediction}.
                Asymptotically valid post-hoc confidence sets are studied by \citet{chugg2026post}.
            \end{rmk}

    \section{From likelihood principle to E-principle}\label{sec:likelihood_principle}
        We propose to use an E-measure (kernel) to present all the relevant evidence, where the relevance of evidence is described by the hypothesis class under consideration.
        We call this the \emph{E-principle}, extending the idea of the likelihood principle to E-measures.
        
        The likelihood principle loosely states that all the relevant evidence about the true $P^*$ is contained in the likelihood function.
        Using the language of E-measures, we can turn this into a precise statement about $\H$-measurability.
        
        In particular, the likelihood itself corresponds to an E-measure kernel for the hypothesis space $(\P, 2^\P)$.
        Indeed, suppose that $\P$ admits a dominating measure $\lambda$.
        Then, the reciprocal of the likelihood $\e$ is the E-kernel density $\e(\{P\} \mid x) = (dP/d\lambda(x))^{-1}$ for the hypothesis class $\H = 2^\P$.
        This E-kernel density can be extended to an E-measure kernel on $\H$ by closure.
        If $\lambda$ is a \emph{probability} measure, then $\e$ is a valid E-kernel.
        
        This yields a formulation of the likelihood principle in terms of E-measures: the likelihood is an E-kernel for the hypothesis class $2^\P$, which is the hypothesis class that makes the map $P \mapsto P$ $\H$-measurable.
        That is, every claim (hypothesis) about the true $P^*$ is $\H$-measurable.

        From this perspective, measurability with respect to a hypothesis class $\H$ can be viewed as a generalization of the likelihood principle: all the relevant evidence is captured by an E-kernel for an appropriately chosen hypothesis class $\H$.
        This also resolves the problem that the classical likelihood principle is hard to extend to non-parametric setting in which $\P$ often does not admit a dominating measure.
        In such settings we can replace the likelihood function by another E-kernel for $2^\P$, or pass to a less granular hypothesis class $\H \subseteq 2^\P$.

    \section{Evidence-based decisions}\label{sec:decisions}
        In this section, we show how E-measures can be used in decision making.

        We consider a decision maker who faces the consequences of making a decision $d \in \D$ under uncertainty about the true $P^*$.
        For a given decision $d$ and potentially true $P \in \P$ the consequences are described by the consequence function $L_P : \D \to \C$, where $\C$ is a space of consequences equipped with a preorder $\succsim$.
        For the sake of interpretation, we consider `large' consequences to be `worse'.
        We denote the collection of all consequence functions by $\L := \{L_P : P \in \P\}$.

        A concrete example of a consequence space is a numerical loss space $[0, \infty]$, equipped with the natural order.
        We return to this special case in Section \ref{sec:grunwald}.

        \subsection{The E-principle in decision making}
            If the decision maker were to observe the true $P^*$, then they could weigh the decisions over their consequences by browsing the true consequence function $L_{P^*}$ over the decisions.
            As the truth is not observed, the best we can do is to present the decision maker with relevant evidence about $P^*$ for the decision problem.

            In theory, we could present the evidence in the shape of an E-measure for the entire power set hypothesis class $\H = 2^\P$.
            However, not all of this is relevant to the decision problem.
            Indeed, the real problem is not that the truth $P^*$ is unknown, but that the true consequence function $L_{P^*}$ is not known.
            As a result, we need only distinguish between elements of $\P$ through their consequences for the decision maker.

            Following the E-principle, the idea is to couple the `relevant information' for a decision problem to $\H$-measurability of the function $P \mapsto L_P$.
            For example, if we wish to present evidence against claims about the true consequence function $L_{P^*}$, the idea is that it suffices that all hypotheses of the form
            \begin{align*}
                \{P \in \P : L_P = \ell\}, \quad \ell \in \L
            \end{align*}
            are $\H$-measurable: contained in $\H$.
            That is, it suffices for the function $P \mapsto L_P$ to be $\H$-measurable, as a function from our hypothesis class $(\P, \H)$ to the space of consequence functions $(\L, 2^\L)$.

            More generally, the kind of evidence we can provide about the consequences of our decisions is only limited by the granularity of the hypothesis class $\H$ for which we have an E-measure.

        \subsection{Hypothesis class for consequence bounds}
            To showcase our ideas, we focus on the problem of obtaining bounds on the consequences of a given decision.
            Here, we derive the hypothesis class $\H_L$ that makes all the lower-bound claims about the true consequences $L_{P^*}$ measurable.
            
            For a given decision $d$, bounds on its consequence $L_{P^*}(d)$ are captured by the hypotheses of the form
            \begin{align*}
                H_{d, c}
                    := \{P \in \P : L_P(d) \succsim c\}, \quad c \in \C.
            \end{align*}
            Here, $P^* \in H_{d, c}$ expresses that the true consequence $L_{P^*}(d)$ of $d$ is at least $c$.

            As we intend to compare across decisions, we consider a hypothesis class that contains all hypotheses $H_{d, c}$ across decisions $d \in \D$ and benchmarks $c \in \C$.
            Proposition \ref{prp:smallest_relevant} shows that the smallest such intersection-closed hypothesis class is $\H_L$, which we define as the union-closure of the hypotheses
            \begin{align*}
                H_\ell
                    := \{P' \in \P : L_{P'}(d) \succsim \ell(d), \textnormal{ for every } d \in \D\}, \quad \ell \in \L.
            \end{align*}
            For a given benchmark consequence function $\ell$, $P^* \in H_\ell$ expresses that the true consequence function $L_{P^*}$ is uniformly worse than $\ell$.
            
            \begin{prp}\label{prp:smallest_relevant}
                $\H_L$ is the smallest intersection-closed hypothesis class that contains every hypothesis $H_{d, c}$, $d \in \D$, $c \in \C$.
            \end{prp}

            The hypothesis class $\H_L$ is not rich enough for $\H_L$-measurability of $P \mapsto L_P$.
            Instead, it makes $P \mapsto L_P$ `order-measurable'; $P \mapsto L_P$ is measurable only with respect to the hypothesis class of upper sets induced by uniform dominance on $\L$.
            
            More precisely, define the preorder $\ell \succsim_\L \ell' \iff \ell(d) \succsim \ell'(d)$ for every $d \in \D$, and let $\H_{\succsim_\L}$ denote the hypothesis class on $\L$ (not $\P$!) induced by this preorder, as in Section \ref{sec:preorder_induced}.
            The map $P \mapsto L_P$ is then $(\H_L, \H_{\succsim_{\L}})$-measurable, yet not necessarily $(\H_L, 2^\L)$-measurable.

            \begin{dfn}[Order-measurable]\label{dfn:order-measurable}
                We say that $P \mapsto L_P$ is $\H$-order-measurable if $P \mapsto L_P$ is $(\H, \H_{\succsim_{\L}})$-measurable.
            \end{dfn}

            \begin{lem}\label{lem:order_measurable}
                $P \mapsto L_P$ is $\H$-order-measurable if and only if $\H_L \subseteq \H$.
            \end{lem}

            \begin{rmk}[Least hypotheses]\label{rmk:HL_least_hypothesis}
                The hypothesis class $\H_L$ is induced by a preorder $\succsim_L$ on $\P$ defined by: $P' \succsim_L P$ if and only if $L_{P'}(d) \succsim L_P(d)$ for every $d \in \D$, because every $\ell \in \L$ corresponds to $L_P$ for some $P \in \P$.
                As a consequence, $(\P, \H_L)$ is intersection-closed with least hypotheses $H_{L_P}$, $P \in \P$, by Proposition \ref{prp:order_least_true}.
            \end{rmk}
            
        \subsection{Evidential consequence bounds}
            Theorem \ref{thm:E-consequence} shows that the validity of an E-kernel $\e$ for $\H_L$ implies an \emph{evidential bound} on the true consequence $L_{P^*}(d)$ uniformly across potentially data-informed decisions.
            To support the interpretation, we use the event-style notation
            \begin{align*}
                \e(L_{P^*}(d) \succsim \ell(d) \mid x)
                    =: \e(H_{d, \ell(d)} \mid x).
            \end{align*}
            For any data-informed decision $\widehat{d}$, the idea is that $\e(L_{P^*}(\widehat{d}) \succsim \ell(\widehat{d}) \mid x)$ expresses the evidence against the claim that the true consequence $L_{P^*}(\widehat{d})$ is at least as bad as $\ell(\widehat{d})$.
            
            \begin{thm}[Uniform E-consequence bound]\label{thm:E-consequence}
                Let $\e$ be a valid E-capacity kernel for $\H$.
                Suppose that $P \mapsto L_P$ is $\H$-order-measurable.
                Then, for every $\ell \in \L$ and $P \in H_\ell$,
                \begin{align*}
                    \Ex^P\left[\sup_{d \in \D} \e(L_{P^*}(d) \succsim \ell(d) \mid X)\right] \leq 1.
                \end{align*}
                Here, we assume the map $x \mapsto \sup_{d \in \D} \e(H_{d, \ell(d)} \mid x)$ is $\Sigma$-measurable.
            \end{thm}
            \begin{proof}
                By Lemma \ref{lem:order_measurable}, $\H$-order-measurability is equivalent to $\H_L \subseteq \H$.
                Hence, both $H_\ell$ and $H_{d, \ell(d)}$ are $\H$-measurable, by Proposition \ref{prp:smallest_relevant}, so that the E-kernel is well-defined on these hypotheses.
                
                Next, $H_\ell = \{P' : L_{P'}(d) \succsim \ell(d) \textnormal{ for every } d \in \D\} \subseteq \{P' : L_{P'}(d) \succsim \ell(d)\} = H_{d, \ell(d)}$.
                Hence, by the antitonicity of E-capacities, we obtain $\e(H_{d, \ell(d)} \mid x) \leq \e(H_\ell \mid x)$.
                Taking the supremum over $\D$ on both sides yields $\sup_{d \in \D} \e(H_{d, \ell(d)} \mid x) \leq \e(H_\ell \mid x)$, so that the result follows from the validity of $\e$.
            \end{proof}
            
            \begin{rmk}
                The $\Sigma$-measurability holds if $\{H_{d, \ell(d)} : d \in \D\}$ is countable.
            \end{rmk}
            
            \begin{rmk}
                To the best of our knowledge, evidential bounds are novel.
                Different notions of validity induce different types of evidential bounds.
                We explore such different notions of validity in Section \ref{sec:multiplicity} and Section \ref{sec:abstract_hypotheses}.
            \end{rmk}

        \subsection{Probabilistic and post-hoc consequence bounds}
            Evidential bounds generalize probabilistic bounds in the same way that E-values generalize binary hypothesis tests.
            To express these bounds, we define the confidence sets on consequences,
            \begin{align*}
                C_\alpha^d(x)
                    := \{c \in \C : \e(H_{d, c} \mid x) < 1 / \alpha\}, \quad d \in \D.
            \end{align*}
            
            Corollary \ref{cor:probability_loss_bound} shows that Theorem \ref{thm:E-consequence} specializes to a probability bound if we consider $\{0, 1/\alpha\}$-valued E-kernels.
            Corollary \ref{cor:post-hoc_probability_loss_bound} is equivalent to Theorem \ref{thm:E-consequence}, but expressed as a post-hoc probability bound, analogously to Proposition \ref{prp:post-hoc}.
            
            \begin{cor}[Probability consequence bound]\label{cor:probability_loss_bound}
                Let $\e$ be an $\{0, 1/\alpha\}$-valued valid E-capacity kernel for $\H$.
                Assume that $P \mapsto L_P$ is $\H$-order-measurable and fix $\alpha \in [0, \infty]$.
                For every $P \in H_\ell$ and $\ell \in \L$,
                \begin{align*}
                    P\left(\ell(d) \in C_\alpha^d(X), \textnormal{ for every } d \in \D\right) \geq 1 - \alpha,
                \end{align*}
                under the measurability assumption of Theorem \ref{thm:E-consequence}.
                In particular, this holds for $\ell(d) = L_P(d)$.
            \end{cor}
            \begin{proof}
                As $\e$ is $\{0, 1/\alpha\}$-valued we have, for every $d \in \D$, $\e(H_{d, \ell(d)} \mid x) = 1/\alpha \times \ind{\e(H_{d, \ell(d)} \mid x) \geq 1/\alpha} = 1/\alpha \times \ind{\ell(d) \not\in C_\alpha^d(x)}$.
                Hence,
                \begin{align*}
                    \sup_{d \in \D} \e(H_{d, \ell(d)} \mid x) 
                        = 1/\alpha \times \ind{\ell(d) \not\in C_\alpha^d(x) \textnormal{ for some } d \in \D}.
                \end{align*}
                The result then follows from Theorem \ref{thm:E-consequence}.
            \end{proof}

            \begin{cor}[Post-hoc consequence bound]\label{cor:post-hoc_probability_loss_bound}
                Let $\e$ be a valid E-capacity kernel for $\H$.
                Assume that $P \mapsto L_P$ is $\H$-order-measurable.
                For every $P \in H_\ell$, $\ell \in \L$ and data-dependent level $\widetilde{\alpha} \in [0, \infty]$,
                \begin{align*}
                    \Ex^P\left[\frac{P\left(\ell(d) \not\in C_{\widetilde{\alpha}}^d(X) \textnormal{ for some } d \in \D \mid \widetilde{\alpha}\right)}{\widetilde{\alpha}}\right] \leq 1,
                \end{align*}
                under the measurability assumption of Theorem \ref{thm:E-consequence}.
                In particular, this holds for $\ell(d) = L_P(d)$.
            \end{cor}

            \begin{rmk}[Related work]\label{rmk:uniform_loss_bounds_literature}
                Analogues of Corollary \ref{cor:probability_loss_bound} and Corollary \ref{cor:post-hoc_probability_loss_bound} have been studied before for real-valued consequences (losses).
                In particular, such a version of Corollary \ref{cor:probability_loss_bound} is introduced by \citet{andrews2025certified}.
                Moreover, such a version of Corollary \ref{cor:post-hoc_probability_loss_bound} is introduced in \citet{koning2025fuzzy}, and of Corollary \ref{cor:probability_loss_bound} by \citet{kiyani2025decision}, both in the context of predictive inference; see Section \ref{sec:abstract_hypotheses}.

                To the best of our knowledge, the idea to generalize beyond real-valued consequences is novel.
            \end{rmk}
            
         \subsection{Relationship to Gr\"unwald's loss bound}\label{sec:grunwald}
            We now present the connection to the work of \citet{grunwald2023posterior}.
            For this purpose, we must move to a numerical consequence (loss) space $\C = [0, \infty]$ equipped with the natural order $\geq$.
            This results in a non-negative $\H$-order-measurable function $P \mapsto L_P(d)$, which is precisely what we need to perform E-integration.
            We develop E-integration in detail in Appendix \ref{sec:E-integration}, but the representation given in Proposition \ref{prp:d_loss_uniform} suffices here.

            Proposition \ref{prp:d_loss_uniform} shows that if $(\P, \H)$ is intersection-closed then the E-integral downweights the consequences $L_P(d)$ against which there exists much evidence.
            Here, the right-hand-side corresponds to the `E-posterior risk assessment' introduced by \citet{grunwald2023posterior}, and minimizing the E-integrated loss corresponds to his `E-posterior-minimax decision rule'.
            
            \begin{prp}\label{prp:d_loss_uniform}
                Let $\e$ be an E-measure for $\H_L$.
                Then, for every $d \in \D$,
                \begin{align*}
                    \int^E L_P(d) \, d\e(P)
                        = \sup_{P \in \P} \frac{L_P(d)}{\e(H_{L_P})}
                        = \sup_{P \in \P} \frac{L_P(d)}{\e(H_{d, L_P(d)})}
                \end{align*}
            \end{prp}
            
            The E-integrated loss is the E-analogue of expected loss.
            As the E-integrated loss is $[0, \infty]$-valued, we can use it to rank the decisions:
            \begin{align*}
                d \mapsto \int^E L_P(d)\, d\e(P).
            \end{align*}
            A natural decision is to minimize the evidence-weighted loss:
            \begin{align*}
                \argmin_{d \in \D} \int^E L_P(d) \, d\e(P),
            \end{align*}
            assuming such a decision exists.
            
            Moving to E-kernels, we recover the main result of \citet{grunwald2023posterior}.
            Its proof follows directly from combining Markov's inequality for E-integrals with Theorem \ref{thm:E-consequence}.

            \begin{thm}[Gr\"unwald-type bound]\label{thm:grunwald_bound}
                Let $\e$ be a valid E-capacity kernel for $\H$.
                Suppose that $P \mapsto L_P$ is $\H$-order-measurable.
                Assume $\D$ is at most countable and $x \mapsto \int^E L_P(d) \, d\e(P \mid x)$ is $\Sigma$-measurable.
                Then, for every $P \in \P$,
                \begin{align*}
                    \Ex^P\left[\sup_{d \in \D}\frac{L_P(d)}{\int^E L_{P'}(d) \, d\e(P' \mid X)}\right] \leq 1.
                \end{align*}
            \end{thm}
            \begin{proof}
                By Proposition \ref{prp:markov} (E-Markov), 
                \begin{align*}
                    \frac{L_P(d)}{\int^E L_{P'}(d) \, d\e(P' \mid x)}
                        &\leq \e\left(L_{P^*}(d) \geq L_P(d) \mid x\right).
                \end{align*}
                Taking the supremum $\sup_{d \in \D}$ on both sides and applying Theorem \ref{thm:E-consequence} yields the result.
            \end{proof}

            \begin{rmk}[Relationship to \citet{grunwald2023posterior}]
                \citet{grunwald2023posterior} explores valid E-function kernels (under the name `E-posterior') for $\H = 2^\P$, as an E-function density on $\P$.
                As the Gr\"unwald-type loss bound passes through E-Markov's inequality, which is as loose as the classical Markov's inequality, it is weaker than our E-consequence bound.
                Moreover, the Gr\"unwald-type bound only works for $[0, \infty]$-valued consequences, whereas the E-consequence bound works for preorder-valued consequences.
                
                An inquiry into why the Gr\"unwald-type bound requires non-negative losses whereas probability bounds as in Corollary \ref{cor:probability_loss_bound} work for arbitrary real-valued losses was the starting point of this work.
            \end{rmk}
            
        \subsection{Evidential admissibility}
            An E-measure for $\H_L$ allows us to compare across decisions through a uniform evidential dominance.
            In particular, we define the preorder $\succsim_{\e}$ on $\D$ through
            \begin{align*}
                d \succsim_{\e} d'
                    \iff \e(L_{P^*}(d) \succsim c) \geq \e(L_{P^*}(d') \succsim c),
                    \quad \textnormal{for every } c \in \C.
            \end{align*}
            This can be interpreted as $d$ being uniformly preferred to $d'$ in terms of its consequences, as far as the evidence is concerned.
            More precisely, for every benchmark consequence $c$, this states that the evidence against the claim that the true consequence under $d$ is at least as bad as $c$ is at least as strong as the corresponding evidence for $d'$.

            Using this preorder on $\D$, we can rule out decisions that are uniformly dominated in terms of the evidence, leading to a notion of evidential admissibility.

            \begin{dfn}[Evidential admissibility]
                A decision $d \in \D$ is evidentially inadmissible if there exists a $d' \in \D$ such that $d' \succ_\e d$ ($\succsim_{\e}$ and $\not\precsim_{\e}$).
                A decision is evidentially admissible if it is not evidentially inadmissible.
            \end{dfn}
            
        \subsection{Identifying an optimal decision and MLE}\label{sec:optimal}
            In this section, we briefly highlight another interesting task in this decision problem: finding an optimal decision.
            
            Suppose a, say unique, optimal decision $d_P^* = \argmin_{d \in \D} L_P(d)$ exists for every $P \in \P$.
            In this case, an alternative task could be to select the optimal decision $d_{P^*}^*$ under the truth $P^*$.
            The relevant notion of measurability here is measurability of this optimal decision $d_{P^*}^*$.
            The smallest hypothesis class for which $d_{P^*}^*$ is measurable is then generated from hypotheses that group probabilities that share the same optimal decision
            \begin{align*}
                H_d := \{P \in \P : d \in \argmin_{d' \in \D} L_P(d')\}.
            \end{align*}
            Here, $P^* \in H_d$ can be interpreted as the claim that the true optimal decision $d_{P^*}^*$ equals $d$.
            A resulting E-measure presents the evidence against a particular decision $d$ being the true optimal decision.
            A decision can then be selected among decisions for which there is little evidence against their optimality.

            \begin{exm}[Maximum likelihood]
                An example is maximum likelihood estimation.
                Indeed, suppose we consider the consequence function $L_{P}(P') = \textnormal{KL}(P \mid P')$, and assume $\P$ admits a dominating $\sigma$-finite measure $\lambda$.
                Then $L_{P}(P') = 0$ only if $P = P'$, so the unique optimal decision is $d_P^* = P$.
                Hence, the hypotheses $H_P = \{P' \in \P : P \in \argmin_{Q \in \P} L_{P'}(Q)\}$ are the singleton sets $\{P\}$.
                This means $\H$-measurability of $d_{P^*}^*$ requires $\H = 2^\P$.
                
                An example of an E-kernel density is the inverse of the likelihood $\e(\{P\} \mid x) = (dP(x)/d\lambda)^{-1}$, which presents the evidence against each $P$ being the truth $P^*$.
                This extends to all hypotheses by closure.
                If $\lambda$ is a probability measure, then $\e$ is a valid E-kernel.
                The decision that minimizes the evidence against the singleton hypotheses $\{P\}$ is the maximum likelihood estimator: $\widehat{P}(x) := \argmin_{P \in \P} \e(\{P\} \mid x) = \argmax_{P \in \P} dP / d\lambda(x)$.

                Using Theorem \ref{thm:E-consequence} we can obtain an E-consequence bound on the true Kullback-Leibler divergence to the maximum likelihood estimator.
                Indeed, for every $P \in \P$,
                \begin{align*}
                    \Ex^P\left[\e\left(\textnormal{KL}(P^* \mid \widehat{P}(X)) \ge \textnormal{KL}(P \mid \widehat{P}(X)) \mid X\right)\right] 
                        \leq 1.
                \end{align*}
            \end{exm}
            
    \section{Cross-hypothesis validity}\label{sec:multiplicity} 
        To this point, we have only considered hypothesis-wise validity of E-kernels.
        We now consider cross-hypothesis notions of validity, which may account for multiplicity.
        
        \subsection{Familywise evidence control}\label{sec:familywise}
            We start with defining `familywise evidence': the E-analogue generalization of the classical `familywise error rate'.
            
            To interpret the familywise evidence, it helps to again consider a `true' probability $P^*$ and to call a hypothesis $H$ `true' if $P^* \in H$.
            Familywise evidence control then ensures we may browse the entire E-kernel $H \mapsto \e(H \mid x)$ to select a hypothesis post-hoc, and be confident that the e-variable against the selected hypothesis is valid if this hypothesis is true.
            Standard validity is only enough to guarantee such validity for prespecified hypotheses.
            
            \begin{dfn}[Familywise evidence]
                Let $\e$ be an E-function kernel.
                Assume that $x \mapsto \sup_{H \ni P} \e(H \mid x)$ is $\Sigma$-measurable, for every $P \in \P$.
                We define the familywise evidence of $\e$ as
                \begin{align*}
                    \sup_{H \in \H : H \ni P} \e(H \mid X),
                \end{align*}
                and say that $\e$ controls the familywise evidence if
                \begin{align*}
                    \sup_{P \in \P} \Ex^P\left[\sup_{H \in \H : H \ni P} \e(H \mid X) \right] \leq 1.
                \end{align*}
            \end{dfn}
            \begin{rmk}
                The $\Sigma$-measurability is guaranteed if $\H$ is at most countable.
            \end{rmk}

            Theorem \ref{thm:familywise} presents our main result for familywise evidence control.
            It states that familywise evidence control coincides with standard validity if $(\P, \H)$ is intersection-closed, for E-capacity kernels.
            This effectively means that we obtain familywise evidence control `for free'.
            
            \begin{thm}[Hypothesis-wise validity = familywise evidence control]\label{thm:familywise}
                Suppose $(\P, \H)$ is intersection-closed.
                Let $\e$ be an E-capacity kernel.
                $\e$ is valid if and only if it controls the familywise evidence. 
            \end{thm}

            \begin{cor}[Closure]
                If $(\P, \H)$ is intersection-closed and $\e$ is a familywise evidence controlling E-capacity kernel, then its closure $\overline{\e}$ is a familywise evidence controlling E-measure kernel that dominates $\e$.
            \end{cor}
            The reason that this result holds is that among all hypotheses $H$ that contain $P$, intersection-closure guarantees the existence of a least hypothesis: $H_P$.
            Since E-capacities are antitone, this least hypothesis automatically attains the largest evidence:
            \begin{align*}
                \sup_{H \ni P} \e(H \mid x)
                    = \e(H_P \mid x).
            \end{align*}
            As a result, controlling the maximum evidence over all true hypotheses for $P$ is equivalent to controlling the evidence for the single (least) true hypothesis $H_P$.

            \begin{rmk}
                The idea to extend the familywise error rate to E-values was briefly sketched in \citet{koning2023post}, and explored in more detail by \citet{hartog2025family}, both for finite collections of hypotheses.
                The fact that hypothesis-wise validity suffices for a wide range of multiplicity-corrections for binary rejection sets was recently established by \citet{xu2025bringing}.
                Our result goes beyond binary rejection sets, by expanding to E-kernels.
            \end{rmk}

            \begin{rmk}[Recovering classical familywise error rate]
                Recall that a classical discovery set $R_\alpha(x)$ is equivalent to an $\{0, 1/\alpha\}$-valued E-function kernel:
                \begin{align*}
                    R_\alpha(x)
                        := \{H \in \H^\emptyset : \e(H \mid x) \geq 1/\alpha\}.
                \end{align*}
                Indeed, we can reconstruct the E-function kernel through $\e(H \mid x) = 1/\alpha \times \ind{H \in R_\alpha(x)}$, $H \in \H^\emptyset$.
                In this binary setting, familywise evidence control is exactly the classical familywise error rate:
                \begin{align*}
                    \sup_{P \in \P} &\frac{P\left(\textnormal{for some } H \in \H, H \ni P : H \in R_\alpha(X)\right)}{\alpha}\\
                        &= \sup_{P \in \P} \Ex^P\left[\sup_{H \in \H : H \ni P} \ind{H \in R_\alpha(X)}/\alpha\right] \\
                            &= \sup_{P \in \P} \Ex^P\left[\sup_{H \in \H : H \ni P} \e(H \mid X)\right]
                            \leq 1.
                \end{align*}
            \end{rmk}

        \subsection{False Evidence Rate control}
            The False Discovery Rate (FDR) is the expected proportion of true hypotheses (false discoveries) in a rejected set \citep{benjamini1995controlling}.
            Generalizing from rejections to evidence, we introduce its natural E-analogue: the False Evidence Rate (FER).
            
            Since we do not have a `rejection set' when we work with a continuous notion of evidence, we replace this with a data-dependent `selection rule'.
            The FER then corresponds to the expected average amount of \emph{evidence} against selected true hypotheses (false evidence).

            \begin{dfn}[Selection rule]
                A selection rule is a map $S : \X \to 2^\H$ such that $S(x)$ is finite, $x \in \X$, and $x \mapsto \ind{H \in S(x)}$ is $\Sigma$-measurable, $H \in \H$.
            \end{dfn}
            
            \begin{dfn}[False evidence proportion and rate]
                Let $\e$ be an E-function kernel and $S$ a selection rule.
                We define the false evidence proportion at $P$ with respect to $S$ as
                \begin{align*}
                    \textnormal{FEP}_P^S(x)
                        := \frac{1}{|S(x)| \vee 1} \sum_{H \in S(x) : H \ni P} \e(H \mid x).
                \end{align*}
                The corresponding false evidence rate is $\sup_{P \in \P} \Ex^P[\textnormal{FEP}_P^S(X)]$, and we say that an E-function kernel controls the FER for $S$ if this false evidence rate is bounded by 1.
            \end{dfn}
            
            \begin{dfn}[False selection proportion]
                For a given $P \in \P$ and selection rule $S$, we denote the proportion of true hypotheses among the selected hypotheses by 
                \begin{align*}
                    \textnormal{FSP}_P^S(x)
                        := \frac{|S(x) \cap \{H : H \ni P\}|}{|S(x)| \vee 1}.
                \end{align*}
                We call this the `false' selection proportion (FSP), because if the selections are treated as `discoveries' then these discoveries are false if $P$ is the truth.
            \end{dfn}

            Our main result here is Theorem \ref{thm:FER}, which provides an upper bound on the false evidence proportion in terms of the evidence against the least hypotheses and False Selection Proportion.
            
            For a given selection rule, Corollary \ref{cor:one_selection_rule} shows that this implies we can get away with a weaker notion of validity than hypothesis-wise validity, since $\textnormal{FSP}_P^S(x) \leq 1$ for all $x \in \X$.
            Corollary \ref{cor:uniform_selection_rule} shows that hypothesis-wise validity is equivalent to uniform validity across selection rules.
            
            \begin{thm}[FER control under intersection-closure]\label{thm:FER}
                Suppose $(\P, \H)$ is intersection-closed.
                Let $\e$ be an E-capacity kernel and $S$ a selection rule.
                Then, for every $P \in \P$ and $x \in \X$,
                \begin{align*}
                    \textnormal{FEP}_P^S(x)
                        \leq \textnormal{FSP}_P^S(x) \e(H_P \mid x)
                        \leq \e(H_P \mid x).
                \end{align*}
            \end{thm}

            \begin{cor}[Fixed selection rule]\label{cor:one_selection_rule}
                Suppose $(\P, \H)$ is intersection-closed.
                If $\e$ is an E-capacity kernel that satisfies
                \begin{align*}
                    \sup_{P \in \P} \Ex^P\left[\textnormal{FSP}_P^S(X) \e(H_P \mid X)\right] \leq 1,
                \end{align*}
                then $\e$ controls the false evidence rate with respect to $S$.
            \end{cor}

            \begin{cor}[Uniform over selection rules]\label{cor:uniform_selection_rule}
                Suppose $(\P, \H)$ is intersection-closed.
                An E-capacity kernel $\e$ is valid if and only if it controls the false evidence rate uniformly across selection rules.
            \end{cor}
            \begin{proof}[Proof of Corollary \ref{cor:uniform_selection_rule}]
                The sufficiency of hypothesis-wise validity follows from observing that the right-hand-side of $\textnormal{FSP}_P^S(x) \e(H_P \mid x) \leq \e(H_P \mid x)$ does not depend on $S$.
                The necessity follows from the existence of the selection rules $S_H(x) = \{H\}$, $H \in \H$, so that for every $P \in H$, $\textnormal{FEP}_P^{S_H}(x) = \e(H \mid x)$.
                As a result, uniform false evidence rate control implies $\Ex^P[\e(H \mid X)] = \Ex^P[\textnormal{FEP}_P^{S_H}(X)] \leq 1$, for every $P \in H$, $H \in \H^\emptyset$.
            \end{proof}
            
            \begin{rmk}[Closure]
                Following Remark \ref{rmk:closure_other_validity}, we can again take the closure of the E-capacity without loss of FER control, so that E-measure kernels are the only non-dominated FER controlling E-capacity kernels.
            \end{rmk}

            \begin{rmk}
                The first inequality in Theorem \ref{thm:FER} is tight in the sense that if the selection rule selects only among least hypotheses and the least hypotheses are disjoint, then the first inequality becomes an equality.
            \end{rmk}

            We can recover the classical False Discovery Rate as a special case of the False Evidence Rate, if we restrict ourselves to $\{0, 1/\alpha\}$-valued E-kernels and use a particular selection rule, as we show in Remark \ref{rmk:recover_FDR}.
            
            \begin{rmk}[Recovering the FDR]\label{rmk:recover_FDR}
                Suppose that our E-kernel is $\{0, 1/\alpha\}$-valued and consider the selection rule $S_\alpha(x) := \{H \in \H : \e(H \mid x) \geq 1/\alpha\}$ that selects hypotheses that are rejected at level $\alpha$.
                Then,
                \begin{align*}
                    \textnormal{FEP}_P(x)
                        = \frac{1}{\alpha} \frac{|S_\alpha(x) \cap \{H : H \ni P\}|}{|S_\alpha(x)| \vee 1}
                        = \frac{1}{\alpha} \textnormal{FSP}_P^{S_\alpha}(x).
                \end{align*}
                As a consequence, false evidence rate control recovers the classical false discovery rate control at level $\alpha$ in this binary setting.

                Since the False Discovery Rate relies on a single selection rule, Corollary \ref{cor:one_selection_rule} shows we can get away with something weaker than hypothesis-wise validity.
                This matches the finding of \citet{ignatiadis2024asymptotic} that `compound E-values' suffice for FDR-control.
                We explore this further in Remark \ref{rmk:compound}.
            \end{rmk}
            
        \subsection{Post-processing E-kernels and (closed) E-BH}
            \begin{figure}[t]
                \centering
                \begin{tikzpicture}[scale=1.05, every node/.style={align=center}]
                    \draw[rounded corners=4pt, thick] (-4.4,-3.6) rectangle (4.4,3.3);
                    \node[anchor=north west] at (-4.25,3.1) {$\P$};
            
                    \def\r{2.05}
                    \coordinate (C1) at (-1.45,0.75);
                    \coordinate (C2) at ( 1.45,0.75);
                    \coordinate (C3) at ( 0.00,-1.20);
            
                    % balanced fills: strong enough to see, weak enough not to swamp overlaps
                    \fill[blue!18, opacity=0.24] (C1) circle (\r);
                    \fill[red!18,  opacity=0.24] (C2) circle (\r);
                    \fill[teal!10, opacity=0.10] (C3) circle (\r);
            
                    % outlines carry most of the color identity
                    \draw[thick, blue!60!black] (C1) circle (\r);
                    \draw[thick, red!60!black]  (C2) circle (\r);
                    \draw[thick, teal!55!black] (C3) circle (\r);
            
                    % labels for the baseline hypotheses
                    \node[text=blue!50!black] at (-2.95, 2.55) {$G_1$};
                    \node[text=red!55!black]  at ( 2.95, 2.55) {$G_2$};
                    \node[text=teal!55!black] at (-1.55,-3.00) {$G_3$};
            
                    % least hypotheses + values
                    \node at (-2.45, 1.35) {$H_1$\\[-1mm] \footnotesize $60$};
                    \node at ( 2.45, 1.35) {$H_2$\\[-1mm] \footnotesize $29$};
                    \node at ( 0.00,-2.55) {$H_3$\\[-1mm] \footnotesize $11$};
            
                    \node at ( 0.00, 1.40) {$H_{12}$\\[-1mm] \footnotesize $70$};
                    \node at (-1.12,-0.62) {$H_{13}$\\[-1mm] \footnotesize $65$};
                    \node at ( 1.12,-0.62) {$H_{23}$\\[-1mm] \footnotesize $40$};
            
                    \node at (0.00,0.12) {$H_{123}$\\[-1mm] \footnotesize $100$};
            
                    \node at (3.35,-2.45) {$H_C$\\[-1mm] \footnotesize $5$};
                \end{tikzpicture}
                \caption{
                    Toy example based on the baseline family of hypotheses $G_1, G_2, G_3$.
                    The eight cells are the least hypotheses $H_1, H_2, H_3, H_{12}, H_{13}, H_{23}, H_{123}, H_C$.
                    The numbers displayed in the cells are the least-hypothesis E-values $\e(H_P)$.
                }
                \label{fig:venn-emeasure-cells}
            \end{figure}
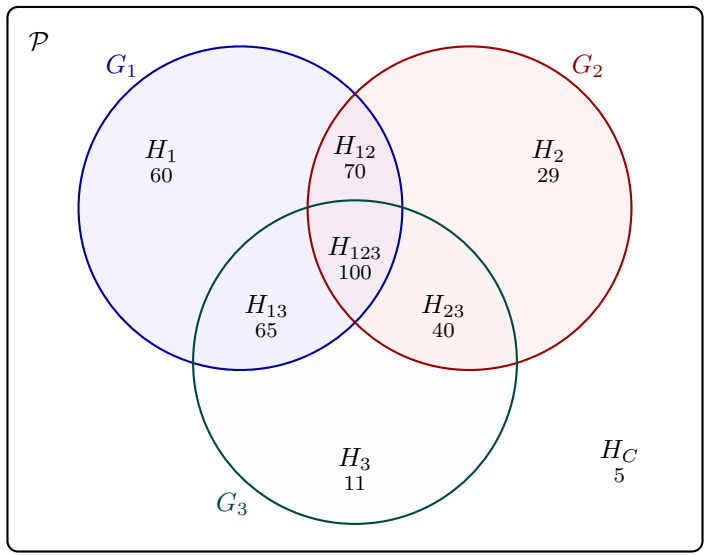
        
            Popular E-value-based multiple testing methods such as the E-BH procedure \citep{wang2022false} and its recent `closed E-BH' improvement \citep{xu2025bringing} effectively post-process an E-kernel into a level-$\alpha$ rejection set.
            We view rejection sets as binary ($\{0, 1/\alpha\}$-valued) E-measures.
            As such, these methods are implicitly converting E-kernels into other E-kernels.
            This conversion is generally lossy, discarding evidence.

            Our approach throughout is of course to directly view the output of the E-kernel as our expression of evidence.
            To showcase this, we compare our E-kernel approach to the E-BH approaches in a toy example.
            In particular, we consider a hypothesis class $\H$ generated by the partition $\{H_1, H_2, H_3, H_{12}, H_{13}, H_{23}, H_{123}, H_C\}$.
            The cells of the partition form the least hypotheses of $\H$.
            We display this hypothesis class in Figure \ref{fig:venn-emeasure-cells}, along with the value of the E-kernel on every cell (the E-kernel density).
            
            For the comparison to the E-BH methods, we think of the hypothesis class $\H$ as the intersections of a special subset of hypotheses $\G = \{G_1, G_2, G_3\}$, defined as $G_1 = H_1 \cup H_{12} \cup H_{13} \cup H_{123}$, $G_2 = H_2 \cup H_{12} \cup H_{23} \cup H_{123}$ and $G_3 = H_3 \cup H_{13} \cup H_{23} \cup H_{123}$.
            We represent $\G$ by the overlapping circles in Figure \ref{fig:venn-emeasure-cells}.

            In the first numerical column of Table \ref{tab:self-consistent-vs-ebh}, we display the E-kernel output over the hypothesis class and over $\G$, as induced by the Closure axiom.
            By Corollary \ref{cor:uniform_selection_rule}, this controls the FER uniformly over all selection rules.

            For a given selection rule $S$, Theorem \ref{thm:FER} shows that FER control only constrains the product $\textnormal{FSP}_P^S \e(H_P)$.
            Hence, we may inflate $\e(H_P)$ by a factor of $1/\textnormal{FSP}_P^S$ while preserving FER control under $S$.
            This means we can post-process a hypothesis-wise valid E-kernel into a FER-valid E-kernel under $S$:
            \begin{align*}
                \e^{S}(H_P)
                    := \e(H_P) / \textnormal{FSP}_P^S.
            \end{align*}
            The value at the remaining hypothesis in $\H$ is obtained by using the closure axiom of the E-measure.

            To showcase the post-processing, we consider a self-referential selection rule as in Remark \ref{rmk:recover_FDR}.
            For the purpose of the example, we constrain the selection rule to the subset of hypotheses $\G$.
            To define an explicit self-referential selection rule, we first define the rejection set for a given selection:
            \begin{align*}
                T(S)
                    := \{G \in \G : \e^S(G) \geq 1 / \alpha\}.
            \end{align*}
            We consider the rule that returns a largest such self-referential selection set,
            \begin{align*}
                S^* \in \argmax\{|S| : S \subseteq \G, S = T(S)\},
            \end{align*}
            which may not be unique.
            
            In our toy example, this selection rule happens to yield a unique largest selection at $\alpha = 0.05$: $S^* = \{G_1, G_2, G_3\}$.
            We show the resulting post-processed E-kernel density in the second numerical column of Table \ref{tab:self-consistent-vs-ebh}.
            
            The final two columns show the output of the E-BH and closed E-BH at $\alpha = 0.05$, also based on post-processing the original E-kernel from the first numerical column.
            Compared to the $S^*$-based E-kernel, both the E-BH and closed E-BH can be viewed as discarding a large amount of evidence.
            In fact, the E-BH is even dominated by the original E-kernel.

            \begin{table}[t]
                \centering
                \footnotesize
                \begin{tabular}{lcrrrr}
                \toprule
                Hypothesis
                  & $\e$
                  & $\e^{S^*}$
                  & $\mathrm{FSP}^{S^*}$
                  & E-BH
                  & closed E-BH \\
                \midrule
                $H_C$     & $5$   & $\infty$ & $0$   & $0$  & $0$  \\
                $H_1$     & $60$  & $180$    & $1/3$ & $20$ & $20$ \\
                $H_2$     & $29$  & $87$     & $1/3$ & $0$  & $20$ \\
                $H_3$     & $11$  & $33$     & $1/3$ & $0$  & $20$ \\
                $H_{12}$  & $70$  & $105$    & $2/3$ & $20$ & $20$ \\
                $H_{13}$  & $65$  & $97.5$   & $2/3$ & $20$ & $20$ \\
                $H_{23}$  & $40$  & $60$     & $2/3$ & $0$  & $20$ \\
                $H_{123}$ & $100$ & $100$    & $1$   & $20$ & $20$ \\
                \midrule
                $G_1$     & $60$  & $97.5$   & --    & $20$ & $20$ \\
                $G_2$     & $29$  & $60$     & --    & $0$  & $20$ \\
                $G_3$     & $11$  & $33$     & --    & $0$  & $20$ \\
                \bottomrule
                \end{tabular}
                \caption{Toy comparison of FER control. The first numerical column is the original hypothesis-wise valid E-kernel density $\e$ from Figure \ref{fig:venn-emeasure-cells}, valid uniformly over all selection rules. The second reports an E-kernel derived from post-processing $\e$ for the greatest self-consistent selected family $S^* \subseteq \G$ at $\alpha = 0.05$. The third captures the associated FSP. The final two columns show the E-kernel representations of E-BH and closed E-BH at $\alpha = 0.05$.}
                \label{tab:self-consistent-vs-ebh}
            \end{table}

            \begin{rmk}[Related work]
                Starting from only a valid E-variable for each of the hypotheses $G_1, G_2, G_3$, we can complete a valid E-kernel on $\H$ by assigning weighted averages of these E-values to the intersections.
                This is suggested in \citet{xu2025bringing} who take such hypotheses $G_1, G_2, G_3$ as a starting point.

                The discussion here is related to the post-hoc choice of the rejection set and the post-hoc choice of $\alpha$ discussed by \citet{xu2025bringing}.
                Perhaps the key conceptual difference is that instead of converting an E-kernel into rejection sets, we advocate for viewing the E-kernel itself as the statistical output.
            \end{rmk}

            \begin{rmk}[Compound E-variables]\label{rmk:compound}
                In this example, compound validity of E-variables \citep{ignatiadis2024asymptotic} corresponds to FER control under the fixed selection rule $S^{\textnormal{all}}(x) = \G$, for every $x \in \X$.
                Indeed, the FEP then becomes
                \begin{align*}
                    \textnormal{FEP}_P^{S^{\textnormal{all}}}(x)
                        = 1/|\G| \sum_{G \in \G : G \ni P} \e(G \mid x),
                \end{align*}
                so that FER control becomes $\sum_{G \in \G: G\ni P} \Ex^P\left[\e(G \mid X)\right] \leq |\G|$, for every $P \in \P$.
                This is exactly the definition of a collection of compoundly valid E-variables for $\G$.
            \end{rmk}

        \subsection{General multiplicity control and \citet{xu2025bringing}}\label{sec:general_multiplicity}
            We now consider a more general notion of multiplicity control for E-kernels, which nests both the FER and Familywise Evidence.
            For this purpose, we use a collection of functions $\Phi_P : [0, \infty]^\H \to [0, \infty]$, $P \in \P$.
            Here, $\Phi_P(\e)$ can be interpreted as the disutility of evidence emitted by $\e$ in case $P$ is the truth.

            \begin{dfn}[General validity]\label{dfn:general_validity}                
                We say that an E-function kernel $\e$ is valid under $(\Phi_P)_{P \in \P}$ if for every $P \in \P$,
                \begin{align*}
                    \Ex^P\left[\Phi_P(\e(\,\cdot \mid X))\right] \leq 1.
                \end{align*}
            \end{dfn}

            Under some mild assumptions on $\Phi_P$, we obtain a clean result that nests both our Familywise Evidence control and the FER results.
            Locality is the natural assumption that our disutility $\Phi_P$ only concerns false evidence: evidence assigned to true hypotheses under $P$.

            \begin{thm}\label{thm:general_multiplicity}
                Suppose $(\P, \H)$ is intersection-closed.
                Assume $\Phi_P$ is:
                \begin{itemize}
                    \item Local. $\Phi_P(\eta) = \Phi_P(\eta 1_P)$, where $1_P(H) := \ind{P \in H}$, $H \in \H$,
                    \item Positively homogeneous. $\Phi_P(c \eta) = c \Phi_P(\eta)$, $c \geq 0$,
                    \item Monotonic. $\eta \geq \eta'$ implies $\Phi_P(\eta) \geq \Phi_P(\eta')$, for every $\eta, \eta' \in [0, \infty]^\H$.
                \end{itemize}
                Then, for every E-capacity $\e$ for $\H$, we have
                \begin{align*}
                    \Phi_P(\e) \leq \e(H_P) \Phi_P(1_P),
                    \quad \textnormal{for every } P \in \P.
                \end{align*}
            \end{thm}

            \begin{cor}\label{cor:general_multiplicity}
                Suppose $(\P, \H)$ is intersection-closed.
                Let $\e$ be an E-capacity kernel for $\H$.
                If for every $P \in \P$,
                \begin{align*}
                    \Ex^P\left[\e(H_P \mid X)\Phi_P(1_P)\right] \leq 1,
                \end{align*}
                then $\e$ is valid under $(\Phi_P)_{P \in \P}$.
            \end{cor}

            \begin{exm}[Recovering Familywise evidence control]
                Familywise evidence control is recovered by choosing
                \begin{align*}
                    \Phi_P(\eta) := \sup_{H \ni P}\eta(H).
                \end{align*}
                This is monotonic, positively homogeneous and local.
                Moreover, $\Phi_P(1_P)=1$ for every $P \in \P$.
                Hence Theorem \ref{thm:general_multiplicity} gives $ \sup_{H \ni P}\e(H \mid x) \leq \e(H_P \mid x)$, for every $P \in \P$.
                As $H_P \ni P$ this holds with equality, so that Corollary \ref{cor:general_multiplicity} recovers Theorem \ref{thm:familywise}. 
            \end{exm}

            \begin{exm}[Recovering FER]
                Let $A \subseteq \H$ be finite, and define
                \begin{align*}
                    \Phi_P^A(\eta)
                        := \frac{1}{|A| \vee 1} \sum_{H \in A : H \ni P} \eta(H),
                \end{align*}
                which matches the false evidence proportion under selection $A$, and is monotonic, positively homogeneous and local.
                Moreover,
                \begin{align*}
                    \Phi_P^A(1_P)
                        = \frac{|A \cap \{H : H \ni P\}|}{|A| \vee 1},
                \end{align*}
                which is the corresponding false selection proportion.
                Theorem \ref{thm:general_multiplicity} then yields Theorem \ref{thm:FER}: $\textnormal{FEP}_P^A(x) \leq \textnormal{FSP}_P^A(x)\e(H_P \mid x)$.

                The uniform FER result can be recovered by choosing $\Phi_P = \sup_A \Phi_P^A$.
            \end{exm}

            \begin{rmk}[Relationship to \citet{xu2025bringing}]
                Let $\R_\alpha(x) \subseteq \H$ be a rejection set.
                We encode this rejection set as a $\{0,1/\alpha\}$-valued E-kernel
                \begin{align*}
                    \e(H \mid x)
                        := \frac{1}{\alpha}\ind{H \in \R_\alpha(x)},
                        \quad H \in \H.
                \end{align*}
                Following \citet{xu2025bringing}, let $f_H : \H \to [0, \infty]$ for every $H \in \H$.
                Define
                \begin{align*}
                    \Phi_P(\eta) := \sup_{S \in \H} f_{H_P}(S) \eta(S),
                \end{align*}
                which is both positively homogeneous and monotonic.
                Then
                \begin{align*}
                    \Phi_P(\e(\,\cdot \mid x))
                        = \sup_{R \in \R_\alpha(x)} f_{H_P}(R) \e(R \mid x)
                        = \sup_{R \in \R_\alpha(x)} \frac{f_{H_P}(R)}{\alpha},
                \end{align*}
                since $\e(R \mid x) = 0$ for $R \not\in \R_\alpha(x)$.
                Validity of $\e$ under $(\Phi_P)$ then becomes
                \begin{align*}
                    \Ex^P\left[\sup_{R \in \R_\alpha(X)} f_{H_P}(R)\right] \leq \alpha,
                    \quad P \in \P,
                \end{align*}
                which is exactly the expected-loss control studied by \citet{xu2025bringing}.

                In our notation, we believe their e-Closure result can be viewed as the statement that $\e$ is valid under $(\Phi_P)$ if and only if there exists another E-kernel $\e'$ that is valid and satisfies
                \begin{align*}
                    \e'(H_P \mid x)  \geq \Phi_P(\e(\cdot \mid x)), \quad \textnormal{ for every } P \in \P,\ x \in \X.
                \end{align*}

                The missing ingredient to apply Theorem \ref{thm:general_multiplicity} and Corollary \ref{cor:general_multiplicity} is locality of $\Phi_P$, which in this representation is recovered by assuming $f_{H_P}(A) = 0$ if $A \not\supseteq H_P$.
            \end{rmk}
    \section{Updating E-measures}
        \subsection{From E-prior to E-posterior}\label{sec:E-posterior}
            For a single hypothesis $H$, a well-known `updating' property of E-variables is that the product $\e_2(H \mid x) = \e_0(H) \e_1(H \mid x)$ of an E-value $\e_0(H)$ and a valid E-variable $x \mapsto \e_1(H \mid x)$ satisfies
            \begin{align*}
                \Ex^P[\e_2(H \mid X)] \leq \e_0(H).
            \end{align*}
            In this section, we generalize this updating rule to E-measures and E-capacities.
            Akin to the Bayesian setting, this allows us to update from `E-prior' to `E-posterior'.
    
            \subsubsection{Updating E-capacities by multiplication}
                We start with an updating rule for E-capacities.
                This updating rule is nothing more than simultaneously updating the individual E-values for each hypothesis.
        
                \begin{prp}[Validity of the E-posterior]\label{prp:E-posterior}
                    Let $\e_0$ be a (prior) E-capacity and $\e_1$ a valid E-capacity kernel.
                    Then, their E-posterior $\e_2(H \mid x) := (\e_0 \times \e_1(\cdot \mid x))(H)$ is an E-capacity kernel that satisfies
                    \begin{align}\label{ineq:validity_E-prior}
                        \sup_{P \in H} \Ex^P[\e_2(H \mid X)] \leq \e_0(H), \quad \textnormal{ for every } H \in \H^\emptyset.
                    \end{align}
                \end{prp}
        
                \begin{rmk}
                    The condition \eqref{ineq:validity_E-prior} can be viewed as validity relative to an E-prior.
                    Standard validity for an E-kernel can be viewed as validity relative to the prior 1-E-measure $\mathbf{1}$, featured in Definition \ref{dfn:1-E-measure} in Appendix \ref{sec:E-integration}.
                \end{rmk}
    
                \begin{rmk}[E-measure-martingale]
                    For a single hypothesis $H$, sequential raw multiplication results in a `test martingale'.
                    Raw multiplication of E-kernels therefore corresponds to a kind of E-measure-martingale.
                \end{rmk}
                
            \subsubsection{Updating E-measures by closed multiplication}
                Unfortunately, the (raw) hypothesis-wise product of two E-measures need not be an E-measure; it is only guaranteed to be an E-capacity.
                Taking the closure of the product produces an E-measure, but the resulting E-measure may not be valid in the sense of \eqref{ineq:validity_E-prior}.
                
                Luckily, we can recover a weaker notion of validity if the hypothesis class is intersection-closed.
                By taking the closure, this updating rule is more than a simple hypothesis-wise update of the e-values.
    
                \begin{prp}[Validity of the closed E-posterior]\label{prp:closed_E-posterior}
                    Suppose $(\P, \H)$ is intersection-closed, with an at most countable collection $\{H_P : P \in \P\}$ of least hypotheses.
                    Let $\e_0$ be a (prior) E-capacity and $\e_1$ a valid E-capacity kernel.
                    Then, their closed E-posterior $\e_2(H \mid x) := \overline{(\e_0 \times \e_1(\cdot \mid x))}(H)$ is an E-measure kernel that satisfies
                    \begin{align*}
                       \Ex^P[\e_2(H \mid X)] \leq \e_0(H_P), \quad \textnormal{ for every } P \in H, H \in \H^\emptyset.
                    \end{align*}
                \end{prp}
    
                \begin{rmk}
                    The closed multiplication procedure can be viewed as performing raw multiplication for the least hypotheses $H_P$, and subsequently using closure to extend to all hypotheses $H \in \H$.
                    This extends an observation of \citet{koning2026anytime}, that for a single composite hypothesis $H$ it is better to track an E-process against every singleton hypothesis $\{P\}$, $P \in H$, than a single E-process against $H$.
                \end{rmk}
                
        \subsection{E-measure processes}
            We can generalize to processes of E-measures. 
            Such processes can be viewed as a generalization of E-processes \citep{grunwald2024safe, ramdas2022admissible} to E-measures.
    
            \begin{dfn}[E-measure-process]
                Let $(\F_t)_{t \geq 0}$ with $\F_t \subseteq \Sigma$ be a filtration on $(\X, \Sigma)$, $t \in \bN$.
                We say that $(\e)_{t \geq 0}$ is an E-measure-process if $H \mapsto \e_t(H \mid x)$ is an E-measure for every $t \geq 0$ and $x$, and $x \mapsto \e_t(H \mid x)$ is $\F_t$-measurable for every $t \geq 0$ and $H$.
            \end{dfn}
    
            \begin{dfn}[Anytime validity]
                An E-measure-process is anytime valid for $\H$ if, for every stopping time $\tau$ adapted to $(\F_t)_{t \geq 0}$, the stopped kernel $\e_{\tau}(H \mid x) := \e_{\tau(x)}(H \mid x)$ is a valid E-kernel.
            \end{dfn}
    
            Theorem \ref{thm:closing_E-process} produces another instantiation of Theorem \ref{thm:closure_least_true_hypothesis}.
            It shows that any anytime valid E-capacity-process is dominated by an anytime valid E-measure-process.
            
            \begin{thm}[Closing E-capacity processes]\label{thm:closing_E-process}
                Suppose $(\P, \H)$ is intersection-closed with an at most countable collection of least hypotheses $\{H_P, P \in \P\}$.
                Let $(\e)_{t \geq 0}$ be an E-capacity process.
                Then, its closure
                \begin{align*}
                    \overline{\e}_t(H \mid x)
                        := \inf_{P \in H} \e_t(H_P \mid x), \quad H \in \H,
                \end{align*}
                defines an E-measure process that dominates $(\e_t)_{t \geq 0}$.
                Moreover, it is anytime valid if and only if $(\e_t)_{t \geq 0}$ is anytime valid.
            \end{thm}
            \begin{proof}
                The proof follows from applying Theorem \ref{thm:closure_pointwise_validity} to the E-kernels $\e_{\tau}$, for every stopping time $\tau$.
            \end{proof}
    
            \begin{rmk}[Relationship to \citet{ramdas2022admissible}]
                Theorem \ref{thm:closing_E-process} can be viewed as an E-measure version of the landmark result of \citet{ramdas2022admissible}, who show that any admissible anytime valid E-process for a $\P = H$ is the (measurable) infimum of anytime valid E-processes for the (least) hypotheses $\{P\}$, $P \in H$.
                Our result also shows that these can be replaced by the potentially smaller canonical cover $\{H_P, P \in H\}$ of $H$.
            \end{rmk}
            
            \begin{rmk}
                In the same way that valid E-kernels are a generalization of valid confidence sets, valid E-measure processes are a generalization of confidence sequences \citep{howard2021time, waudbysmith2023estimating}.
                Moreover, they are equivalent to post-hoc confidence sequences.
            \end{rmk}
        
    \section{Abstracting the notion of a `hypothesis'}\label{sec:abstract_hypotheses}
        For the sake of presentation, we have followed the convention in the statistics literature that a \emph{hypothesis} is a collection of probabilities.
        We now break this convention, and consider reporting evidence against more abstract notions of a hypothesis.
        Almost everything we developed generalizes: the only point at which we use the fact that $\H$ consists of sets of probabilities is in defining notions of validity.

        In particular, we can retain a model $\P$ and a sample space equipped with a $\sigma$-algebra $(\X, \Sigma)$, but replace the hypothesis space by $(\Y, \H)$, where $\Y$ is some other space.
        The only challenge is to characterize an appropriate notion of validity.
        
        \subsection{Predictive E-kernels}\label{sec:prediction}
            To showcase the abstraction of a hypothesis, we consider a prediction setting in which the hypothesis class $\H$ is also a $\sigma$-algebra satisfying $\H \supseteq \Sigma$.
            This means our hypothesis space is a measurable space $(\X, \H)$, equipped with a model $\P$.
            Here, the $\sigma$-algebra $\Sigma$ may be interpreted as the observed information, and $\H$ as a richer `full-information' $\sigma$-algebra.

            A hypothesis $H \in \H$ can be interpreted as a predictive claim about the `full outcome' $x^* \in \X$.
            The hypothesis $H$ is `true' at outcome $x$ if $x \in H$.
            We express our predictions about $x^*$ through evidence against the statement $x^* \in H$ for every hypothesis $H \in \H$, using a predictive E-kernel $\e : \H \times \X \to [0, \infty]$.

            The definition of predictive validity is analogous to familywise evidence control as in Section \ref{sec:familywise}, except that the supremum ranges over hypotheses that contain the realized outcome $X$ instead of $P$.
            
             \begin{dfn}[Predictive validity]
                A predictive E-kernel $\e$ is valid if
                \begin{align*}
                    \Ex^P\left[\sup_{H \in \H : H \ni X} \e(H \mid X)\right] \leq 1,\quad \textnormal{ for every } P \in \P.
                \end{align*}
            \end{dfn}

            \begin{rmk}
                Restricting to $\{0, 1/\alpha\}$-valued E-kernels recovers prediction sets analogously to the confidence sets from Section \ref{sec:confidence_sets}.
            \end{rmk}

            We automatically have that $(\X, \H)$ is intersection-closed, as it is both a $\sigma$-algebra and union closed.
            This suffices to show that the validity of a predictive E-kernel is determined by a single $\H$-measurable valid E-variable for $\P$.
            %See Definition \ref{dfn:measurable_function} for the relevant definition of $\H$-measurability.

            \begin{thm}[Single E-variable suffices]\label{thm:predictive_kernel_least_true}
                Suppose $(\X, \H)$ is intersection-closed, with least hypotheses $H_x$, $x \in \X$.
                Let $\e$ be a predictive E-capacity kernel.
                Then, for every $x \in \X$,
                \begin{align*}
                    \sup_{H \in \H : H \ni x} \e(H \mid x) 
                        = \e(H_x \mid x).
                \end{align*}
                
                As a consequence, $\e$ is valid if and only if $x \mapsto \e(H_x \mid x)$ is a valid $\H$-measurable E-variable for $\P$.
            \end{thm}

            \begin{rmk}[Recovering E-kernel from E-variable?]
                The single E-variable $x \mapsto \e(H_x \mid x)$ is enough to characterize predictive validity, but not to reconstruct the full predictive E-kernel on $(\X, \H)$, even if we assume it is an E-measure kernel.
                Indeed, for a fixed outcome $x$, reconstructing $\e(H \mid x)$ requires the values $\e(H' \mid x)$ for all least hypotheses $H' \subseteq H$; not only the single value $\e(H_x \mid x)$.
            \end{rmk}

            \begin{rmk}[$\H = 2^\X$]
                In predictive inference, it is standard to (implicitly) consider a hypothesis class $\H = 2^\X$.
                Here, the least hypotheses are the singleton sets $H_x = \{x\}$, $x \in \X$.
                Theorem \ref{thm:predictive_kernel_least_true} then reduces to the statement that a predictive E-capacity kernel is valid if and only if $x \mapsto \e(\{x\} \mid x)$ is a valid $\H$-measurable E-variable for $\P$, which recovers one of the key insights of \citet{koning2025fuzzy}.
            \end{rmk}

        \subsection{Pushing forward E-kernels}
            A simple way to construct E-measures and E-kernels on a target space is to push them forward along a measurable map.

            Suppose we have two hypothesis spaces $(\Y, \H)$ and $(\Z, \G)$.
            Moreover, suppose we have measurable map $f : \Y \to \Z$, meaning that $f^{-1}(G) \in \H$, for every $G \in \G$.
            Given an E-measure $\e$ on $(\Y, \H)$, we can then define its pushforward E-measure $\e_f$ on $(\Z, \G)$ through
            \begin{align*}
                \e_f(G)
                    := \e(f^{-1}(G)), \quad G \in \G.
            \end{align*}

            \begin{lem}\label{lem:pushforward_E-measure}
                If $\e$ is an E-measure on $(\Y, \H)$ then $\e_f$ is an E-measure on $(\Z, \G)$.
            \end{lem}

            We can analogously define the pushforward of an E-kernel.
            Taking $\Y = \P$, we can pushforward a notion of validity for an E-kernel $\e$ on $(\P, \H)$ to $(\Z, \G)$.

            \begin{dfn}[Pushforward validity]
                Let $\e$ be an E-kernel on $(\P, \H)$.
                Then, we say that $\e_f$ is valid if
               \begin{align*}
                   \Ex^{P}\left[\e_f(G \mid X)\right] \leq 1, \quad \textnormal{for every } P \in \P \textnormal{ and } G \in \G \textnormal{ with } f(P) \in G.
               \end{align*}
            \end{dfn}
           
            \begin{lem}
                If $\e$ is valid on $(\P, \H)$ then its pushforward $\e_f$ is valid on $(\Z, \G)$.
            \end{lem}
            \begin{proof}
                If $f(P) \in G$, then $P \in f^{-1}(G)$.
                Hence, by the validity of $\e$, $\Ex^P[\e_f(G \mid X)]  = \Ex^P[\e(f^{-1}(G) \mid X)]  \leq 1$.
            \end{proof}
            
            \begin{exm}[Parameter space]
                Suppose that $\theta : \P \to \Theta$ is some map into a parameter space $\Theta$.
                We can then pushforward an E-kernel on $\P$ to the parameter space through $\e_\theta(A) := \e(\{P \in \P : \theta(P) \in A\})$, $A \subseteq \Theta$.
            \end{exm}

            \begin{exm}[Evidence against optimality]
                Following Section \ref{sec:optimal}, suppose we have a unique optimal decision $d_P^* := \argmin_{d \in \D} L_P(d)$ for every $P \in \P$.
                We can then define the map $f : \P \to \D$ as returning this optimal decision $f(P) = d_P^*$.
                We can subsequently pushforward an E-kernel on $(\P, \H)$ to $(\D, 2^\D)$ through
                \begin{align*}
                    \e_f(A \mid x)
                        := \e(\{P \in \P : d_P^* \in A\} \mid x), \quad A \subseteq \D.
                \end{align*}
                For every hypothesis $A \subseteq \D$, this expresses evidence against the claim that the true optimal decision is in $A$.
            \end{exm}

            \begin{exm}[Consequence function space]
                Recall the space $\L$ of consequence functions from Section \ref{sec:decisions}.
                We equip $\L$ with the hypothesis class $\H_\succsim$, generated from the uniform-dominance preorder:
                \begin{align*}
                    \ell \succsim \ell' \iff \ell(d) \succsim \ell'(d), \quad \textnormal{ for every } d \in \D.
                \end{align*}
                That is, from the upper set hypotheses on $\L$: $U_\ell := \{\ell' \in \L : \ell' \succsim \ell\}$.
                Consider the function $f(P) = L_P$.
                Then we get $f^{-1}(\H_\succsim) = \H_L$, recovering exactly the hypothesis class $\H_L$ that is central in Section \ref{sec:decisions}.
            \end{exm}

    \section{Acknowledgments}
        I thank Ruben van Beesten for a conversation about decision making under uncertainty, Junu Lee for several conversations about multiple testing with E-values, Sam van Meer for pointing out that the existence of a least hypothesis is equivalent to intersection-closure, and Stan Koobs for general comments.
        
    \bibliographystyle{abbrvnat}
    \bibliography{bibliography}

@article{grunwald2024beyond,
  title={Beyond Neyman--Pearson: E-values enable hypothesis testing with a data-driven alpha},
  author={Gr{\"u}nwald, Peter D},
  journal={Proceedings of the National Academy of Sciences},
  volume={121},
  number={39},
  pages={e2302098121},
  year={2024},
  publisher={National Academy of Sciences}
}

@article{chugg2026post,
  title={Post-Hoc Large-Sample Statistical Inference},
  author={Chugg, Ben and Gauthier, Etienne and Jordan, Michael I and Ramdas, Aaditya and Waudby-Smith, Ian},
  journal={arXiv preprint arXiv:2603.08002},
  year={2026}
}

@article{marcus1976closed,
    author = {Marcus, Ruth and Eric, Peritz and Gabriel, K. R.},
    title = {On closed testing procedures with special reference to ordered analysis of variance},
    journal = {Biometrika},
    volume = {63},
    number = {3},
    pages = {655-660},
    year = {1976},
    month = {12},
    issn = {0006-3444},
    doi = {10.1093/biomet/63.3.655},
}

@article{koning2024continuous,
  title={Continuous testing: Unifying tests and e-values},
  author={Koning, Nick W},
  journal={arXiv preprint arXiv:2409.05654},
  year={2024}
}

@article{ramdas2025hypothesis,
  title={Hypothesis testing with e-values},
  author={Ramdas, Aaditya and Wang, Ruodu},
  journal={Foundations and Trends{\textregistered} in Statistics},
  volume={1},
  number={1-2},
  pages={1--390},
  year={2025},
  publisher={Emerald Publishing Limited}
}

@article{shafer2021testing,
    author = {Shafer, Glenn},
    title = {Testing by Betting: A Strategy for Statistical and Scientific Communication},
    journal = {Journal of the Royal Statistical Society Series A: Statistics in Society},
    volume = {184},
    number = {2},
    pages = {407-431},
    year = {2021},
    month = {05},
    issn = {0964-1998},
    doi = {10.1111/rssa.12647}
}

@article{vovk2021evalues,
  title={E-values: Calibration, combination and applications},
  author={Vovk, Vladimir and Wang, Ruodu},
  journal={The Annals of Statistics},
  volume={49},
  number={3},
  pages={1736--1754},
  year={2021},
  publisher={Institute of Mathematical Statistics}
}

@article{waudbysmith2023estimating,
    author = {Waudby-Smith, Ian and Ramdas, Aaditya},
    title = {Estimating means of bounded random variables by betting},
    journal = {Journal of the Royal Statistical Society Series B: Statistical Methodology},
    volume = {86},
    number = {1},
    pages = {1-27},
    year = {2023},
    month = {02},
    issn = {1369-7412},
    doi = {10.1093/jrsssb/qkad009},
}

@article{howard2021time,
  title={Time-uniform, nonparametric, nonasymptotic confidence sequences},
  author={Howard, Steven R and Ramdas, Aaditya and McAuliffe, Jon and Sekhon, Jasjeet},
  journal={The Annals of Statistics},
  volume={49},
  number={2},
  pages={1055--1080},
  year={2021},
  publisher={JSTOR}
}

@article{koobs2026equivalence,
  title={Equivalence testing with data-dependent and post-hoc equivalence margins},
  author={Koobs, Stan and Koning, Nick W},
  journal={arXiv preprint arXiv:2603.16213},
  year={2026}
}

@inproceedings{kiyani2025decision,
  title={Decision Theoretic Foundations for Conformal Prediction: Optimal Uncertainty Quantification for Risk-Averse Agents},
  author={Kiyani, Shayan and Pappas, George J and Roth, Aaron and Hassani, Hamed},
  booktitle={International Conference on Machine Learning},
  pages={30943--30965},
  year={2025},
  organization={PMLR}
}

@article{andrews2025certified,
  title={Certified decisions},
  author={Andrews, Isaiah and Chen, Jiafeng},
  journal={arXiv preprint arXiv:2502.17830},
  year={2025}
}

@article{xu2025bringing,
  title={Bringing closure to false discovery rate control: A general principle for multiple testing},
  author={Xu, Ziyu and Solari, Aldo and Fischer, Lasse and de Heide, Rianne and Ramdas, Aaditya and Goeman, Jelle},
  journal={arXiv preprint arXiv:2509.02517},
  year={2025}
}

@article{wang2022false,
    author = {Wang, Ruodu and Ramdas, Aaditya},
    title = {False Discovery Rate Control with E-values},
    journal = {Journal of the Royal Statistical Society Series B: Statistical Methodology},
    volume = {84},
    number = {3},
    pages = {822-852},
    year = {2022},
    month = {01},
    doi = {10.1111/rssb.12489},
}

@article{koning2026anytime,
    author = {Koning, Nick W and van Meer, Sam},
    title = {Anytime validity is free: inducing sequential tests},
    journal = {Journal of the Royal Statistical Society Series B: Statistical Methodology},
    pages = {qkag050},
    year = {2026},
    month = {02},
    doi = {10.1093/jrsssb/qkag050},
}

@article{benjamini1995controlling,
  title={Controlling the false discovery rate: a practical and powerful approach to multiple testing},
  author={Benjamini, Yoav and Hochberg, Yosef},
  journal={Journal of the Royal statistical society: series B (Methodological)},
  volume={57},
  number={1},
  pages={289--300},
  year={1995},
  publisher={Wiley Online Library}
}

@article{ignatiadis2024asymptotic,
  title={Asymptotic and compound e-values: multiple testing and empirical Bayes},
  author={Ignatiadis, Nikolaos and Wang, Ruodu and Ramdas, Aaditya},
  journal={arXiv preprint arXiv:2409.19812},
  year={2024}
}

@article{hartog2025family,
  title={Family-wise error rate control with e-values},
  author={Hartog, Will and Lei, Lihua},
  journal={arXiv preprint arXiv:2501.09015},
  year={2025}
}

@article{ramdas2022admissible,
  title={Admissible anytime-valid sequential inference must rely on nonnegative martingales},
  author={Ramdas, Aaditya and Ruf, Johannes and Larsson, Martin and Koolen, Wouter},
  journal={arXiv preprint arXiv:2009.03167},
  year={2022}
}

@article{grunwald2024safe,
  title={Safe testing},
  author={Gr{\"u}nwald, Peter and de Heide, Rianne and Koolen, Wouter},
  journal={Journal of the Royal Statistical Society. Series B: Statistical Methodology},
  volume={86},
  number={5},
  pages={1136--1137},
  year={2024},
  publisher={Oxford University Press}
}

@article{koning2025fuzzy,
  title={Fuzzy prediction sets: Conformal prediction with e-values},
  author={Koning, Nick W and van Meer, Sam},
  journal={arXiv preprint arXiv:2509.13130},
  year={2025}
}

@article{koning2023post,
  title={Post-hoc alpha Hypothesis Testing and the Post-hoc p-value},
  author={Koning, Nick W},
  journal={arXiv preprint arXiv:2312.08040},
  year={2023}
}

@inproceedings{shilkret1971maxitive,
  title={Maxitive measure and integration},
  author={Shilkret, Niel},
  booktitle={Indagationes Mathematicae (Proceedings)},
  volume={74},
  pages={109--116},
  year={1971},
  organization={Elsevier}
}

@article{grunwald2023posterior,
  title={The e-posterior},
  author={Gr{\"u}nwald, Peter D},
  journal={Philosophical Transactions of the Royal Society A: Mathematical, Physical and Engineering Sciences},
  volume={381},
  number={2247},
  year={2023},
  publisher={The Royal Society}
}
    
    \appendix

    \section{E-integration}\label{sec:E-integration}
        To aggregate over an E-function, we consider a notion \emph{E-integration} with respect to an E-function $\e$.
        E-integrals share some properties with classical integrals: positive homogeneity, monotonicity (for E-capacities), their behavior on indicator functions (inverted), and point evaluation under Dirac E-measures.

        The key difference to classical integrals is that E-integrals allow the interchange of supremum and E-integration.
        This replaces the interchange of convex averages and integrals for classical measures.
        The E-integral corresponds to the Shilkret integral with respect to $p = 1/\e$ \citep{shilkret1971maxitive}.

        Recall order-measurability from Definition \ref{dfn:order-measurable}, where we consider the natural order $\geq$ on $[0, \infty]$.
        
        \begin{dfn}[$\H$-order-measurable]\label{dfn:measurable_function}
            A function $f : \mathcal{P} \to [0, \infty]$ is said to be $\mathcal{H}$-order-measurable if $\{f \geq c\} \in \mathcal{H}$ for all $c > 0$.
        \end{dfn}

        \begin{lem}[Order-measurable supremum]\label{lem:measurable_supremum}
            The supremum $\sup_{f \in \F} f$ of a class $\F$ of $\H$-order-measurable functions $f \in \F$ is $\H$-order-measurable.
        \end{lem}
        
        % UNUSED
        % \begin{lem}
        %     If $f$ is measurable then $\{f > c\}$ is measurable for all $c \geq 0$.
        % \end{lem}
        % \begin{proof}
        %     This follows from $\{f > c\} = \bigcup_{b > c} \{f \geq b\}$ and union-closure of $\H$.
        % \end{proof}
        
        \begin{dfn}[E-integral]
            Let $f : \P \to [0, \infty]$ be an $\mathcal{H}$-order-measurable function.
            The E-integral of $f$ with respect to an E-function $\e$ is defined as
            \begin{align*}                
                \int^E f d\e = \sup_{c > 0} \frac{c}{\e(\{f \geq c\})}.
            \end{align*}
        \end{dfn}
        
        \begin{lem}[Positive homogeneity]\label{lem:pos_hom}
            For every $a > 0$, $\int^E af d\e = a\int^E f d\e$.
        \end{lem}
        
        \begin{lem}[Indicator]\label{lem:indicator}
            We have $\int^E 1_H d\e = 1/\e(H)$, for every $H \in \H$.
        \end{lem}

        \begin{lem}[Monotonicity]\label{lem:monotonicity}
            If $\e$ is an E-capacity and $f \geq g$ then $\int^E f d\e \geq \int^E g d\e$.
        \end{lem}

        \begin{lem}[Interchanging E-integral and supremum]\label{lem:sup_interchange}
            Let $\F$ be a collection of $\H$-order-measurable functions.
            If $\e$ is an E-capacity, $\int^E \sup_{f \in \F} f d\e
                    \geq \sup_{f \in \F} \int^E f d\e$.
            If $\e$ is an E-measure, $\int^E \sup_{f \in \F} f d\e
                    = \sup_{f \in \F} \int^E f d\e$.
        \end{lem}

        \begin{dfn}[Dirac E-measure]
            The (inverse of) a Dirac measure is an E-measure: $1 / \delta_P(H) := 1 / \ind{P \in H}$.
        \end{dfn}
        
        \begin{lem}[Dirac integral]\label{lem:dirac}
            We have $\int^E f d(1/\delta_P) = \sup_{c > 0} c \ind{f \geq c} = f(P)$.
        \end{lem}

        \begin{dfn}[1-E-measure]\label{dfn:1-E-measure}
            We define the constant `1-E-measure' $\mathbf{1} : \H \to [0, \infty]$ as $\mathbf{1}(H) = 1$ for $H \in \H^\emptyset$ and $\mathbf{1}(\emptyset) = \infty$.
        \end{dfn}

        \begin{lem}\label{lem:1-E-measure-integral}
            We have $\int^E f d\mathbf{1}= \sup_{P \in \P} f(P)$.
        \end{lem}

        \subsection{E-Integration under intersection-closure}
            If $(\P, \H)$ is intersection-closed, E-integration with respect to an E-measure can be expressed as a supremum over $P \in \P$ instead of $c > 0$.
            Here, we use the convention $0 / 0 = 0$.
            
            \begin{prp}\label{prp:least_true_integration}
                Suppose $(\P, \H)$ is intersection-closed with least hypotheses $H_P$, $P \in \P$.
                Let $\e$ be an E-measure for $\H$.
                Let $\F$ be a collection of $\H$-order-measurable functions $f : \P \to [0,\infty]$.
                Then,
                \begin{align*}
                    \int^E \sup_{f \in \F} f \, d\e
                        = \sup_{f \in \F} \sup_{P \in \P} \frac{f(P)}{\e(H_P)}
                \end{align*}
            \end{prp}
            
        \subsection{E-Markov and Post-hoc E-Markov}
            There exists an analogue to Markov's inequality for E-integrals.
            The post-hoc Markov's equality \citep{koning2023post} for E-integrals simply recovers E-integration.
            \begin{prp}[E-Markov]\label{prp:markov}
                For every $c > 0$ and $\H$-order-measurable $f : \P \to [0, \infty]$, we have
                \begin{align*}
                    \int^E f d\e \geq c / \e(\{f \geq c\}).
                \end{align*}
            \end{prp}
            \begin{proof}
                This follows directly from the definition of E-integration: $\int^E f d\e = \sup_{b > 0} b / \e(\{f \geq b\}) \geq c / \e(\{f \geq c\})$, for every $c > 0$.
            \end{proof}
    
            \begin{prp}[Post-hoc E-Markov]\label{prp:post-hoc_markov}
                For every $\H$-order-measurable $f : \P \to [0, \infty]$, we have
                \begin{align*}
                    \int^E f d\e
                        = \int^E \sup_{c > 0} c\ind{f \geq c} d\e
                        = \sup_{c > 0} \int^E c\ind{f \geq c} d\e
                        = \sup_{c > 0} \frac{c}{\e(\{f \geq c\})}.
                \end{align*}
            \end{prp}
            \begin{proof}
                The first equality follows from the identity $\sup_{c > 0} c\ind{f(P) \geq c} = f(P)$, $P \in \P$ (Lemma \ref{lem:dirac}).
                The second equality follows from interchanging supremum and E-integration (Lemma \ref{lem:sup_interchange}).
                The third equality follows from positive homogeneity (Lemma \ref{lem:pos_hom}) and the behavior on indicators (Lemma \ref{lem:indicator}).
            \end{proof}

            \begin{rmk}
                Note the relationship to the classical Markov's inequality $\int f\, dP \geq cP(\{f \geq c\})$, which is recovered by replacing $\e$ by $1/P$ and $\int^E f\, d\e$ by $\int f\, dP$.
            \end{rmk}
            
   \section{Canonical extension to the power set}\label{appn:canonical_extension}
        Due to the antitonicity of an E-capacity $\e$, its value $\e(H_P)$ at a least hypothesis $H_P$ can be interpreted as the `most evidence' against $P$.
        Using the closure property of E-measures, we can use this idea to obtain a canonical extension of an E-measure $\e$ to the full power set $2^\P$ of $\P$:
        \begin{align*}
            \e^\P(H)
                := \inf_{P \in H} \e(H_P), \quad \textnormal{ for every } H \subseteq \P.
        \end{align*}
        This extension is canonical in the sense that it is the smallest E-capacity extension of $\e$ to $2^\P$.

        \begin{prp}[Extension to $2^\P$]\label{prp:extension}
            Suppose $(\P, \H)$ is intersection-closed.
            Let $\e$ be an E-measure on $\H$.
            Then,
            \begin{itemize}
                \item $\e^\P$ is an E-measure on $2^\P$,
                \item $\e^\P$ coincides with $\e$ on $\H$,
                \item $\e^\P$ is dominated by any E-capacity that coincides with $\e$ on $\H$.
            \end{itemize}
        \end{prp}
        \begin{proof}
            To show that $\e^\P$ is an E-measure, note that for every $\A \subseteq 2^\P$,
            \begin{align*}
                \e^\P\left(\bigcup_{A \in \A} A\right)
                    = \inf_{P \in \cup_{A \in \A} A} \e(H_P)
                    = \inf_{A \in \A} \inf_{P \in A} \e(H_P)
                    = \inf_{A \in \A} \e^\P(A).
            \end{align*}
            
            To show that it extends $\e$ on $\H$, let $H \in \H$.
            By Lemma \ref{lem:canonical_cover}, $H = \cup_{P \in H} H_P$.
            Then, since $\e$ is an E-measure,
            \begin{align*}
                \e(H)
                    = \e\left(\bigcup_{P \in H} H_P\right)
                    = \inf_{P \in H} \e(H_P)
                    = \e^\P(H).
            \end{align*}

            Finally, we show that any E-capacity that extends $\e$ dominates $\e^\P$.
            Let $\widetilde{\e}$ be some E-capacity on $2^\P$ that coincides with $\e$ on $\H$.
            Fix a subset $A \subseteq \P$ and define
            \begin{align*}
                U_A := \bigcup_{P \in A} H_P.
            \end{align*}
            
            As $\H$ is closed under unions, $U_A \in \H$.
            Since $\widetilde{\e}$ coincides with $\e$ on $\H$, this implies $\widetilde{\e}(U_A) = \e(U_A)$.
            Now, because $\e$ is an E-measure, this yields
            \begin{align}\label{eq:proof_extension}
                \widetilde{\e}(U_A)
                    = \e(U_A)
                    = \inf_{P \in A} \e(H_P)
                    = \e^\P(A).
            \end{align}

            As $P \in H_P$ for every $P \in \P$, we have that $A \subseteq U_A$.
            The fact that $\widetilde{\e}$ is an E-capacity on $2^\P$ therefore implies $\widetilde{\e}(A) \geq \widetilde{\e}(U_A)$.
            Combining this with \eqref{eq:proof_extension} yields
            \begin{align*}
                \widetilde{\e}(A) \geq \e^\P(A).
            \end{align*}
            Since $A \subseteq \P$ was arbitrary, this proves the final claim.
        \end{proof}
        \begin{rmk}
            We stress that this extension $\e^\P$ does not add any information compared to the original E-measure $\e$: it is merely a different representation of $\e$ if $(\P, \H)$ is intersection-closed.
        \end{rmk}

    \section{Omitted proofs}
        \subsection{Proof of Lemma \ref{lem:intersection_closed}}
            \begin{proof}
                First suppose that in $(\P,\H)$ every $P \in \P$ has a least hypothesis $H_P$.
                Since $P \in H_P$ for every $P \in \P$, union closure implies
                \begin{align*}
                    \P = \bigcup_{P \in \P} H_P \in \H.
                \end{align*}

                Now let $\S \subseteq \H$ and write $I := \bigcap_{H \in \S} H$.
                If $P \in I$, then $P \in H$ for every $H \in \S$.
                By leastness, $H_P \subseteq H$ for every $H \in \S$, and hence $H_P \subseteq I$.
                This shows
                \begin{align*}
                \bigcup_{P \in I} H_P \subseteq I.
                \end{align*}
                The reverse inclusion follows from $P \in H_P$ for every $P \in I$, so that $I = \bigcup_{P \in I} H_P$.
                As every $H_P \in \H$ and $\H$ is union-closed, it follows that $I \in \H$.
                Hence, $\H$ is intersection-closed.

                For the converse, suppose that $\H$ is closed under arbitrary intersections and $\P \in \H$.
                For every $P \in \P$, define
                \begin{align*}
                    H_P := \bigcap_{H \in \H : H \ni P} H.
                \end{align*}
                This is well-defined since $\P \in \H$ and so every $P$ is an element of some $H \in \H$.
                Moreover, $H_P \in \H$ by intersection-closure.
                Clearly $P \in H_P$, and if $P \in H \in \H$, then $H_P \subseteq H$ by construction.
                Hence $H_P$ is the least hypothesis containing $P$.
            \end{proof}

        \subsection{Proof of Lemma \ref{lem:canonical_cover}}
            \begin{proof}
                First, let us assume that $(\P, \H)$ is intersection-closed with least hypotheses $\{H_P : P \in \P\}$.
                By definition, we have $P \in H_P$.
                Now, fix a hypothesis $H \in \H$.
                If $P \in H$, then $H_P \subseteq H$ by leastness of $H_P$, so
                \begin{align*}
                    \bigcup_{P \in H} H_P \subseteq H.
                \end{align*}
                On the other hand, if $P' \in H$ then $P' \in H_{P'}$ so that
                \begin{align*}
                    \bigcup_{P \in H} H_P \supseteq H.
                \end{align*}
                Combining the two yields equality.

                For the converse, suppose that $H_P \in \H$ and $P \in H_P$ for every $P \in \P$, and that
                \begin{align}\label{eq:union}
                    H = \bigcup_{P' \in H} H_{P'}, \quad H \in \H.
                \end{align}
                Fix $P$ and let $H \in \H$ be some arbitrary hypothesis that contains $P$.
                Since $P \in H$, the set $H_P$ appears as one of the sets in the union \eqref{eq:union}, so that $H_P \subseteq H$.
                As $P \in H_P$, this shows that $H_P$ is the least hypothesis in $\H$ that contains $P$.
                Since this holds for every $P \in \P$, every $P$ has a least hypothesis in $\H$.
                By Lemma \ref{lem:intersection_closed}, we have that $(\P,\H)$ is intersection-closed.
            \end{proof}
            
        \subsection{Proof of Proposition \ref{prp:order_least_true}}
            \begin{proof}
                We first prove the `$\impliedby$'-direction: suppose that $\H = \H_\precsim$, for some preorder $\precsim$.
                
                Fix $P \in \P$, and note that $P \in H_\precsim(P)$ by reflexivity of $\precsim$.
                We now show that $H_\precsim(P)$ is contained in every hypothesis $H \in \H_\precsim$ that contains $P$.
                Let $H \in \H_\precsim$ be such a hypothesis that contains $P$.
                By construction of $\H_\precsim$, $H$ is a union of generating hypotheses:
                \begin{align*}
                    H = \bigcup_{P' \in A} H_\precsim(P'),
                \end{align*}
                for some $A \subseteq \P$.
                Since $P \in H$, there must exist some $P' \in A$ such that $P \in H_\precsim(P')$.
                By definition of $H_\precsim(P')$, this means $P' \precsim P$.

                Now, pick an arbitrary $Q \in H_\precsim(P)$.
                Then, $P \precsim Q$, so that by transitivity
                \begin{align*}
                    P' \precsim P \precsim Q.
                \end{align*}
                Hence, $Q \in H_\precsim(P')$.
                Since $Q$ was arbitrary, this holds for every $Q \in H_\precsim(P)$ and so
                \begin{align*}
                    H_\precsim(P) 
                        \subseteq H_\precsim(P') 
                        \subseteq H.
                \end{align*}
                Since $H$ was an arbitrary hypothesis in $\H_\precsim$ containing $P$, every hypothesis $H \in \H_\precsim$ that contains $P$ must contain $H_\precsim(P)$.
                That is, $H_\precsim(P)$ is the least hypothesis in $\H_\precsim$ containing $P$.
                By Lemma \ref{lem:intersection_closed}, this implies $(\P, \H)$ is intersection-closed.

                We now prove the `$\implies$'-direction: suppose that $(\P, \H)$ is intersection-closed, with least hypotheses $H_P$, $P \in \P$.
                Define the relation $\precsim_\H$ on $\P$ by
                \begin{align*}
                    P \precsim_\H P'
                        \iff P' \in H_P.
                \end{align*}
                
                Reflexivity of $\precsim_\H$ now follows from $P \in H_P$.
                For transitivity, suppose that $P \precsim_\H P'$ and $P' \precsim_\H Q$.
                Then, $P' \in H_P$ and $Q \in H_{P'}$.
                Since $P' \in H_P$, the leastness of $H_{P'}$ implies $H_{P'} \subseteq H_P$.
                Hence $Q \in H_P$, and so $P \precsim_\H Q$.
                Combining reflexivity and transitivity, we conclude that $\precsim_\H$ is a preorder.

                To show that $\H = \H_{\precsim_\H}$, note that $H_{\precsim_\H}(P) = \{P' \in \P : P \precsim_\H P'\} = \{P' \in \P : P' \in H_P\} = H_P$ for every $P \in \P$.
                Hence, the generating family of $\H_{\precsim_\H}$ coincides with the least hypotheses.
                By Lemma \ref{lem:canonical_cover}, $\H$ is the union-closure of the least hypotheses, so that $\H = \H_{\precsim_\H}$.
            \end{proof}
    
        \subsection{Proof of Proposition \ref{prp:induced_E-measure}}
            \begin{proof}
                Consider an arbitrary collection of hypotheses $\S \subset \H$, and define $U := \cup_{H \in \S} H$.
                To prove that $\overline{\e}$ is an E-measure, we need to show that $\overline{\e}(U) = \inf_{H \in \S} \overline{\e}(H)$.
                We proceed by proving both $\leq$ and $\geq$ to conclude equality.

                The ``$\leq$''-direction is straightforward.
                For every $H \in \S$ we have $H \subseteq U$, so that every cover of $U$ is also a cover of $H$.
                The definition of $\overline{\e}$ then implies $\overline{\e}(U) \leq \overline{\e}(H)$.
                Taking the infimum over $H \in \S$ yields
                \begin{align*}
                    \overline{\e}(U)
                        \leq \inf_{H \in \S} \overline{\e}(H).
                \end{align*}

                We continue with the ``$\geq$''-direction.
                We fix an arbitrary number $a < \inf_{H \in \S} \overline{\e}(H)$.
                Our strategy is to show that $\overline{\e}(U) \geq a$.
                As this can be done for every such $a$, we can conclude $\overline{\e}(U) \geq \inf_{H \in \S} \overline{\e}(H)$.

                For each $H \in \S$ we have $a < \overline{\e}(H)$, since $a < \inf_{H \in \S} \overline{\e}(H)$.
                By the supremum in the definition of $\overline{\e}$, this means there must be at least one cover $\T_H \in \C(H)$ such that $a < \inf_{H' \in \T_H} \e(H')$; otherwise $\overline{\e}(H) = \sup_{\T \in \C(H)} \inf_{H' \in \T} \e(H') \leq a$.

                Now define the union $\T := \bigcup_{H \in \S} \T_H$ of such covers.
                As every $\T_H$ is a cover of $H$, this union $\T$ covers $U$: $\T \in \C(U)$.
                Moreover, every hypothesis $H^\dag \in \T$ lies in some cover $\T_{H^1}$.
                For such a cover with $H^\dag \in \T_{H^1}$, we have
                \begin{align*}
                    \e(H^\dag) 
                        \geq \inf_{H' \in \T_{H^1}} \e(H') > a.
                \end{align*}
                As a consequence, $\inf_{H^\dag \in \T} \e(H^\dag) \geq a$.
                Now, since $\overline{\e}(U)$ is the supremum over all covers of $U$, and $\T$ is such a cover, we obtain $\overline{\e}(U) \geq a$.
                Since this holds for arbitrary $a < \inf_{H \in \S} \overline{\e}(H)$, we finish the ``$\geq$''-direction: $\overline{\e}(U) \geq \inf_{H \in \S} \overline{\e}(H)$.
            \end{proof}
            
        \subsection{Proof of Theorem \ref{thm:smallest_dominating_E-measure}}
            \begin{proof}
                We first show that $\overline{\e}$ dominates $\e$. 
                For any $H \in \H$, the singleton $\S = \{H\}$ belongs to $\C(H)$, so that
                \begin{align*}
                    \overline{\e}(H) \geq \inf_{H'\in \{H\}} \e(H') = \e(H).
                \end{align*}

                Now let $\widetilde{\e}$ be an arbitrary E-measure that dominates $\e$: $\widetilde{\e} \geq \e$.
                Fix a hypothesis $H \in \H$, and consider a cover $\T \in \C(H)$ of $H$.
                Since $H \subseteq \cup_{H' \in \T} H'$ and every E-measure is an E-capacity, we have
                \begin{align*}
                    \widetilde{\e}(H)
                        \geq \widetilde{\e}\left(\bigcup_{H' \in \T} H'\right)
                        = \inf_{H' \in \T} \widetilde{\e}(H')
                        \geq \inf_{H' \in \T} \e(H').
                \end{align*}
                Taking the supremum over all $\T \in \C(H)$ yields $\widetilde{\e}(H) \geq \overline{\e}(H)$.
                Hence, $\overline{\e}$ is the smallest E-measure that dominates $\e$.
            \end{proof}

        \subsection{Proof of Theorem \ref{thm:closure_least_true_hypothesis}}
            \begin{proof}
                Our strategy is to show both `$\geq$' and `$\leq$' to prove equality.

                For the $\geq$-direction, Lemma \ref{lem:canonical_cover} allows us to express intersection-closure as
                \begin{align*}
                    H = \bigcup_{P \in H} H_P.
                \end{align*}
                This means that the collection $\T_H := \{H_P : P \in H\}$ is a cover of $H$: $\T_H \in \C(H)$.
                By definition of the closure $\overline{\e}$, this implies
                \begin{align*}
                    \overline{\e}(H) 
                        \equiv \sup_{\T \in \C(H)} \inf_{H' \in \T} \e(H')
                        \geq \inf_{H' \in \T_H} \e(H')
                        = \inf_{P \in H} \e(H_P).
                \end{align*}

                For the $\leq$-direction, it is more convenient to use Lemma \ref{lem:intersection_closed}.
                Let $\T \in \C(H)$ be some cover of $H$.
                For every $P \in H$, we then have that there exists some $H' \in \T$ such that $P \in H'$.
                Lemma \ref{lem:intersection_closed} implies $H_P \subseteq H'$.
                Applying the antitonicity property of the capacity $\e$, we obtain $\e(H') \leq \e(H_P)$.
                Now, since $\inf_{H^\dag \in \T} \e(H^\dag) \leq \e(H')$, this implies $\inf_{H^\dag \in \T} \e(H^\dag) \leq \e(H_P)$, for every $P \in H$.
                Taking the infimum over $P \in H$ yields $\inf_{H^\dag \in \T} \e(H^\dag) \leq \inf_{P \in H} \e(H_P)$.
                Finally, as this holds for every cover $\T \in \C(H)$, taking the supremum over all covers yields
                \begin{align*}
                    \overline{\e}(H)
                        \equiv \sup_{\T \in \C(H)} \inf_{H^\dag \in \T} \e(H^\dag) 
                        \leq \inf_{P \in H} \e(H_P).
                \end{align*}
            \end{proof}

        \subsection{Proof of Corollary \ref{cor:closure_least_true}}
            \begin{proof}
                We first show that $\overline{\e}(H_P) = \e(H_P)$, by showing both $\geq$ and $\leq$.
                By Theorem \ref{thm:closure_least_true_hypothesis}, we have that
                \begin{align*}
                    \overline{\e}(H_P)
                        = \inf_{P' \in H_P} \e(H_{P'}). 
                \end{align*}
                If $P' \in H_P$, then by the leastness of $H_{P'}$, we have $H_{P'} \subseteq H_P$.
                Since $\e$ is an E-capacity, $\e(H_{P'}) \geq \e(H_P)$, so that
                \begin{align*}
                    \inf_{P' \in H_P} \e(H_{P'}) 
                        \geq \e(H_P).
                \end{align*}
                Conversely, as $P \in H_P$, we have that $P$ appears in the infimum, so that
                \begin{align*}
                    \inf_{P' \in H_P} \e(H_{P'}) \leq \e(H_P).
                \end{align*}

                For the uniqueness, suppose that $\widetilde{\e}$ is some E-measure that coincides with $\e$ on the least hypotheses.
                Let $H$ be an arbitrary hypothesis $H \in \H$.
                By Lemma \ref{lem:canonical_cover}, $H = \bigcup_{P \in H} H_P$.
                Hence, by the E-measure property and coincidence on least hypotheses,
                \begin{align*}
                    \widetilde{\e}(H)
                        = \widetilde{\e}\left(\bigcup_{P \in H} H_P\right)
                        = \inf_{P \in H} \widetilde{\e}(H_P)
                        = \inf_{P \in H} \e(H_P).
                \end{align*}
                From Theorem \ref{thm:closure_least_true_hypothesis}, we have $\inf_{P \in H} \e(H_P) = \overline{\e}(H)$.
                Hence, $\widetilde{\e}(H) = \overline{\e}(H)$, for every $H \in \H$.
            \end{proof}

        \subsection{Proof of Proposition \ref{prp:merging}}
            \begin{proof}
                Let $(\e_i)_{i \in \I}$ denote the collection of E-capacity kernels and denote the convex average by
                \begin{align*}
                    \e = \sum_{i \in \I} \lambda_i \e_i,
                \end{align*}
                for a collection of weights $(\lambda_i)_{i \in \I}$, $\sum_{i \in \I} \lambda_i = 1$, $\lambda_i \geq 0$, with an at-most countable index set $\I$.
                
                Fix a hypothesis $H \in \H$.
                As the E-variable $x \mapsto \e_i(H \mid x)$ is $\Sigma$-measurable for every $i$, and at most countable non-negatively weighted sums of measurable functions are measurable, $\e$ is also $\Sigma$-measurable.
                The validity of $\e$ follows from the validity of an at-most-countable weighted average of the E-variables $(\e_i(\cdot \mid x))_{i \in \I}$, which follows from the closure of expectations under such weighted averages.
                
                Now, for the E-capacity property, fix $x \in \X$.
                As every $\e_i(\cdot \mid x)$ is an E-capacity, we have $\e_i(\emptyset \mid x) = \infty$ and, whenever $H \subseteq H'$, $\e_i(H' \mid x) \leq \e_i(H \mid x)$.
                This implies both $\e(\emptyset \mid x) = \sum_{i \in \I} \lambda_i \e_i(\emptyset \mid x) = \infty$ and
                \begin{align*}
                    \e(H' \mid x)
                        = \sum_{i \in \I} \lambda_i \e_i(H' \mid x)
                        \leq \sum_{i \in \I} \lambda_i \e_i(H \mid x)
                        = \e(H \mid x),
                \end{align*}
                so that $\e(\cdot \mid x)$ is also an E-capacity.    
            \end{proof}

        \subsection{Proof of Proposition \ref{prp:post-hoc}}
            \begin{proof}
                Fix $H \in \H^\emptyset$ and $P \in H$.
                By definition of $C_\alpha$, we have that for every $\alpha \in [0, \infty]$,
                \begin{align*}
                    \ind{H \not\in C_\alpha(X)}
                        = \ind{\e(H \mid X) \geq 1/\alpha}.
                \end{align*}

                We first prove the `only if'-direction: suppose that $\e$ is valid and let $\widetilde{\alpha} \in [0, \infty]$ be some data-dependent level.
                Then,
                \begin{align*}
                    \frac{\ind{H \not\in C_{\widetilde{\alpha}}(X)}}{\widetilde{\alpha}}
                        = \frac{\ind{\e(H \mid X) \geq 1/\widetilde{\alpha}}}{\widetilde{\alpha}}
                        \leq \e(H \mid X).
                \end{align*}
                Taking expectations, using the tower property, linearity of expectations and $\sigma(\widetilde{\alpha})$-measurability of $\widetilde{\alpha}$ yields
                \begin{align*}
                    \Ex^P\left[\frac{P(H \not\in C_{\widetilde{\alpha}} \mid \widetilde{\alpha})}{\widetilde{\alpha}}\right]
                        &= \Ex^P\left[\frac{\Ex^P[\ind{H \not\in C_{\widetilde{\alpha}}(X)} \mid \widetilde{\alpha}]}{\widetilde{\alpha}}\right] \\
                        &= \Ex^P\left[\Ex^P\left[\frac{\ind{H \not\in C_{\widetilde{\alpha}}(X)}}{\widetilde{\alpha}}\, \middle| \, \widetilde{\alpha}\right]\right] \\
                        &= \Ex^P\left[\frac{\ind{H \not\in C_{\widetilde{\alpha}}(X)}}{\widetilde{\alpha}}\right]
                        \leq \Ex^P[\e(H \mid X)]
                        \leq 1.
                \end{align*}

                For the converse, suppose that \eqref{ineq:post-hoc_valid_confidence_set} holds for every data-dependent level $\widetilde{\alpha} \in [0, \infty]$.
                This means it holds in particular for the choice $\widetilde{\alpha} = 1 / \e(H \mid X)$, for which we have
                \begin{align*}
                    \frac{\ind{H \not\in C_{\widetilde{\alpha}}(X)}}{\widetilde{\alpha}}
                        = \frac{\ind{\e(H \mid X) \geq 1/\widetilde{\alpha}}}{\widetilde{\alpha}}
                        = \frac{\ind{\e(H \mid X) \geq \e(H \mid X)}}{1/\e(H \mid X)}
                        = \e(H \mid X).
                \end{align*}
                As a consequence,
                \begin{align*}
                    \Ex^P[\e(H \mid X)]
                        = \Ex^P\left[\frac{\ind{H \not\in C_{\widetilde{\alpha}}(X)}}{\widetilde{\alpha}}\right]
                        = \Ex^P\left[\frac{P(H \not\in C_{\widetilde{\alpha}} \mid \widetilde{\alpha})}{\widetilde{\alpha}}\right]
                        \leq 1,
                \end{align*}
                so that $x \mapsto \e(H \mid x)$ is a valid E-variable for $H$.

                The result then follows from the fact that $H \in \H^\emptyset$ and $P \in H$ were arbitrary.
            \end{proof}

        \subsection{Proof of Proposition \ref{prp:smallest_relevant}}
            \begin{proof}
                By Remark \ref{rmk:HL_least_hypothesis}, $\H_L = \H_{\succsim_L}$.
                This implies that $(\P, \H_L)$ is intersection-closed by Proposition \ref{prp:order_least_true}.
                
                We now show that $H_{d,c} \in \H_L$ for every $d \in \D$ and $c \in \C$.
                Since $\H_L = \H_{\succsim_L}$, it suffices to show that $H_{d,c}$ is an upper set for $\succsim_L$.
                Let $P \in H_{d,c}$ and suppose that $P' \succsim_L P$.
                Then $L_{P'}(d) \succsim L_P(d) \succsim c$ and so $P' \in H_{d,c}$.
                Hence $H_{d,c}$ is an upper set for $\succsim_L$, so $H_{d,c} \in \H_L$.

                It remains to show that $\H_L$ is the smallest hypothesis class with these properties.
                Let $\widetilde{\H}$ be any hypothesis class that contains every $H_{d,c}$ and assume that $(\P, \widetilde{\H})$ is intersection-closed.
                By Proposition \ref{prp:order_least_true}, $\widetilde{\H} = \H_{\succsim}$ for some preorder $\succsim$ on $\P$.
            
                Since every set in $\H_{\succsim}$ is an upper set for $\succsim$, each $H_{d,c}$ is an upper set for $\succsim$. 
                Hence, if $P' \succsim P$, then for every $d \in \D$ we have $P \in H_{d,L_P(d)}$, so also $P' \in H_{d,L_P(d)}$, which implies
                \begin{align*}
                    L_{P'}(d) \succsim L_P(d), \qquad \textnormal{for every } d \in \D.
                \end{align*}
                This shows that $P' \succsim P$ implies $P' \succsim_L P$. 
                This implies that an upper set for $\succsim_L$ is also an upper set for $\succsim$, so that $\H_L = \H_{\succsim_L} \subseteq \H_{\succsim} = \widetilde{\H}$.
                As a consequence, $\H_L$ is the smallest such hypothesis class.
            \end{proof}

        \subsection{Proof of Lemma \ref{lem:order_measurable}}
            \begin{proof}
                For convenience, we write $f(P) := L_P$.
                For every $\ell \in \L$, the principal upper set $U_\ell := \{\ell' \in \L : \ell' \succsim_\L \ell\}$ satisfies
                \begin{align*}
                    f^{-1}(U_\ell) = \{P \in \P : L_P \succsim_\L \ell\} = H_\ell.
                \end{align*}
                Hence, if $f$ is $(\H,\H_{\succsim_\L})$-measurable, then $H_\ell \in \H$ for every $\ell \in \L$, and therefore $\H_L \subseteq \H$ by closure under unions.
            
                For the converse, assume $\H_L \subseteq \H$.
                Every $A \in \H_{\succsim_\L}$ is the union of the principal upper sets it contains: $A = \bigcup_{\ell \in A} U_\ell$.
                As a consequence,
                \begin{align*}
                    f^{-1}(A) = \bigcup_{\ell \in A} f^{-1}(U_\ell) = \bigcup_{\ell \in A} H_\ell \in \H_L \subseteq \H.
                \end{align*}
                Hence, $f$ is $(\H,\H_{\succsim_\L})$-measurable.
            \end{proof}

        \subsection{Proof of Proposition \ref{prp:d_loss_uniform}}
            \begin{proof}
                Fix $d \in \D$.
                For every $c > 0$, we have $H_{d,c} = \{P' \in \P : L_{P'}(d) \geq c\} \in \H_L$ by Proposition \ref{prp:smallest_relevant}, so that $P \mapsto L_P(d)$ is $\H_L$-measurable.

                As $(\P, \H_L)$ is intersection-closed, with least hypotheses $H_{L_P}$, $P \in \P$, we can apply Proposition \ref{prp:least_true_integration} to the function $P \mapsto L_P(d)$:
                \begin{align*}
                    \int^E L_P(d) \, d\e(P)
                        = \sup_{P \in \P} \frac{L_P(d)}{\e(H_{L_P})}.
                \end{align*}

                Again using Proposition \ref{prp:smallest_relevant}, we have $H_{d,L_P(d)} \in \H_L$ for every $P \in \P$.
                Since $\e$ is an E-measure, Theorem \ref{thm:closure_least_true_hypothesis} implies
                \begin{align*}
                    \e(H_{d,L_P(d)})
                        = \inf_{P' \in H_{d,L_P(d)}} \e(H_{L_{P'}}).
                \end{align*}
                As a consequence,
                \begin{align*}
                    \sup_{P \in \P} \frac{L_P(d)}{\e(H_{d, L_P(d)})}
                        = \sup_{P \in \P} \sup_{P' \in H_{d, L_P(d)}} \frac{L_P(d)}{\e(H_{L_{P'}})}.
                \end{align*}
                By definition of $H_{d, L_P(d)}$, $P' \in H_{d, L_P(d)}$ is equivalent to $L_{P'}(d) \geq L_P(d)$.
                Hence,
                \begin{align*}
                    \sup_{P \in \P} \sup_{P' \in H_{d, L_P(d)}} \frac{L_P(d)}{\e(H_{L_{P'}})}
                        &= \sup_{P \in \P} \sup_{P' : L_{P'}(d) \geq L_{P}(d)} \frac{L_P(d)}{\e(H_{L_{P'}})} \\
                        &= \sup_{P, P': L_{P'}(d) \geq L_{P}(d)} \frac{L_P(d)}{\e(H_{L_{P'}})} \\
                        &= \sup_{P' \in \P} \sup_{P : L_{P'}(d) \geq L_{P}(d)} \frac{L_P(d)}{\e(H_{L_{P'}})} \\
                        &= \sup_{P' \in \P} \frac{L_{P'}(d)}{\e(H_{L_{P'}})}.
                \end{align*}
            \end{proof}

        \subsection{Proof of Theorem \ref{thm:familywise}}
            \begin{proof}
                Let $H_P$, $P \in \P$ denote the least hypotheses of $\H$.
                By definition of least, we have that if $H \ni P$ then $H_P \subseteq H$.
                By antitonicity of E-capacities, this implies
                \begin{align*}
                    \e(H \mid x) \leq \e(H_P \mid x).
                \end{align*}
                Since $H_P \ni P$, the supremum is attained at $H_P$:
                \begin{align*}
                    \sup_{H \ni P} \e(H \mid x) = \e(H_P \mid x).
                \end{align*}
                Moreover, $x \mapsto \sup_{H \ni P} \e(H \mid x)$ is measurable since $x \mapsto \e(H_P \mid x)$ is measurable.
                Hence, familywise evidence control is equivalent to
                \begin{align}\label{ineq:proof_familywise}
                    \sup_{P \in \P} \Ex^P\left[\e(H_P \mid X)\right] \leq 1.
                \end{align}

                Now, if $\e$ is valid, then by definition $x \mapsto \e(H_P \mid x)$ is valid for $H_P$, since $P \in H_P$.
                Taking the supremum over $P \in \P$ then proves familywise evidence control.
                
                Conversely, suppose $\e$ controls the familywise evidence.
                Consider a hypothesis $H \ni P$.
                Since $H \supseteq H_P$, antitonicity implies $\e(H \mid x) \leq \e(H_P \mid x)$, for every $P \in H$.
                Taking the expectation under $P$ and using \eqref{ineq:proof_familywise} implies the validity of $\e$.
            \end{proof}

        \subsection{Proof of Theorem \ref{thm:FER}}
            \begin{proof}
                Fix $x \in \X$ and $P \in \P$.
                By Lemma \ref{lem:intersection_closed}, $H \ni P$ implies $H_P \subseteq H$.
                By the antitonicity of E-capacities, $\e(H \mid x) \leq \e(H_P \mid x)$.
                As a consequence,
                \begin{align*}
                    \textnormal{FEP}_P^S(x)
                        \leq \frac{1}{|S(x)| \vee 1} \sum_{H \in S(x) : H \ni P} \e(H_P \mid x).
                \end{align*}
                Now, $\sum_{H \in S(x) : H \ni P} \e(H_P \mid x) = |S(x) \cap \{H : H \ni P\}| \e(H_P \mid x)$, since $\e(H_P \mid x)$ is constant in $H \ni P$.
                This implies
                \begin{align*}
                    \frac{1}{|S(x)| \vee 1} \sum_{H \in S(x) : H \ni P} \e(H_P \mid x)
                        &= \frac{|S(x) \cap \{H : H \ni P\}|}{|S(x)| \vee 1} \e(H_P \mid x) \\
                        &= \textnormal{FSP}_P^S(x) \e(H_P \mid x).
                \end{align*}
            \end{proof}

        \subsection{Proof of Theorem \ref{thm:general_multiplicity}}
            \begin{proof}
                Fix $P \in \P$.
                By locality, $\Phi_P(\e) = \Phi_P(\e 1_P)$.
                By antitonicity of E-capacities,
                \begin{align*}
                    \e(H') 1_P(H') \leq \e(H_P) 1_P(H'),
                    \quad H' \in \H,
                \end{align*}
                since both sides are $0$ if $H' \not\ni P$, while if $H' \ni P$ then $H_P \subseteq H'$ and hence $\e(H_P) \geq \e(H')$.
                Hence, $\e 1_P \le \e(H_P) 1_P$.
                By monotonicity and positive homogeneity of $\Phi_P$, we obtain
                \begin{align*}
                    \Phi_P(\e)
                        = \Phi_P(\e 1_P)
                        \leq \Phi_P(\e(H_P) 1_P)
                        = \e(H_P) \Phi_P(1_P).
                \end{align*}
            \end{proof}
            
        \subsection{Proof of Proposition \ref{prp:E-posterior}}
            \begin{proof}
                For every $H \in \H$, the map
                \begin{align*}
                    x \mapsto \e_2(H \mid x) = \e_0(H)\e_1(H \mid x)
                \end{align*}
                is $\Sigma$-measurable.
                For every $x \in \X$, the map $H \mapsto \e_2(H \mid x)$ is an E-capacity, since it is the pointwise product of two E-capacities.
                
                Now fix $H \in \H^\emptyset$.
                Since $\e_0(H)$ does not depend on $x$,
                \begin{align*}
                    \sup_{P \in H} \Ex^P[\e_2(H \mid X)]
                        &= \sup_{P \in H} \Ex^P[\e_0(H)\e_1(H \mid X)] \\
                        &= \e_0(H)\sup_{P \in H} \Ex^P[\e_1(H \mid X)] \\
                        &\leq \e_0(H),
                \end{align*}
                because $\e_1$ is valid.
            \end{proof}
            
        \subsection{Proof of Proposition \ref{prp:closed_E-posterior}}
            \begin{proof}
                The map $H \mapsto \e_0(H) \e_1(H \mid x)$ is an E-capacity, as $\e_0$ and $H \mapsto \e_1(H \mid x)$ are E-capacities.
                Theorem \ref{thm:closure_least_true_hypothesis} then implies the closure equals
                \begin{align}\label{eq:proof_closed_E-posterior}
                    \e_2(H \mid x)
                        = \inf_{P' \in H} \e_0(H_{P'})\e_1(H_{P'} \mid x).
                \end{align}
                By the countability assumption, measurability follows exactly as in Theorem \ref{thm:closure_pointwise_validity}, so $\e_2$ is an E-measure kernel.
            
                Now fix $H \in \H^\emptyset$ and $P \in H$.
                Since $P$ is one of the indices in the infimum in \eqref{eq:proof_closed_E-posterior}, $\e_2(H \mid x) \leq \e_0(H_P)\e_1(H_P \mid x)$.
                Taking expectations under $P$ gives
                \begin{align*}
                    \Ex^P[\e_2(H \mid X)]
                        &\leq \e_0(H_P)\Ex^P[\e_1(H_P \mid X)] \\
                        &\leq \e_0(H_P),
                \end{align*}
                since $P \in H_P$ and $\e_1$ is valid.
            \end{proof}

        \subsection{Proof of Theorem \ref{thm:predictive_kernel_least_true}}
            \begin{proof}
                The equality $\sup_{H \in \H : H \ni x} \e(H \mid x) = \e(H_x \mid x)$ follows directly from the definition of $H_x$ and the fact that $\e$ is an E-capacity.
                Indeed, this implies $\e(H \mid x) \leq \e(H_x \mid x)$ for every hypothesis $H \supseteq H_x$, so that the supremum is attained because $x \in H_x$.

                For the $\H$-measurability, fix $c \geq 0$.
                Then,
                \begin{align*}
                    \{x : \e(H_x \mid x) \geq c\}
                        &= \left\{x : \sup_{H \in \H : H \ni x} \e(H \mid x) \geq c \right\} \\
                        &= \bigcup_{H \in \H} \left(H \cap \{x : \e(H \mid x) \geq c\}\right).
                \end{align*}
                For every $H \in \H$, $x \mapsto \e(H \mid x)$ is $\Sigma$-measurable, and so
                \begin{align*}
                    \{x : \e(H \mid x) \geq c\} 
                        \in \Sigma 
                        \subseteq \H.
                \end{align*}
                Since $\H$ is a $\sigma$-algebra, $H \cap \{x : \e(H \mid x) \geq c\} \in \H$.
                As $\H$ is also a hypothesis class, it is closed under arbitrary unions, so that $\{x : \e(H_x \mid x) \geq c\} \in \H$.
                As this holds for every $c \geq 0$, $x \mapsto \e(H_x \mid x)$ is $\H$-measurable.
                
                The validity claim follows directly from the proven equality and the respective definitions of validity.
            \end{proof}

        \subsection{Proof of Lemma \ref{lem:pushforward_E-measure}}
            \begin{proof}
                Let $\S \subseteq \G$.
                Since preimages preserve unions, we have
                \begin{align*}
                    f^{-1}\left(\bigcup_{G \in \S} G\right)
                        = \bigcup_{G \in \S} f^{-1}(G).
                \end{align*}
                Hence,
                \begin{align*}
                    \e_f\left(\bigcup_{G \in \S} G\right)
                        &= \e\left(f^{-1}\left(\bigcup_{G \in \S} G\right)\right)
                        = \e\left(\bigcup_{G \in \S} f^{-1}(G)\right) \\
                        &= \inf_{G \in \S} \e(f^{-1}(G))
                        = \inf_{G \in \S} \e_f(G),
                \end{align*}
                so that $\e_f$ is an E-measure on $(\Z,\G)$.
            \end{proof}
            
        \subsection{Proof of Lemma \ref{lem:measurable_supremum}}
            \begin{proof}
                This follows from $\left\{\sup_{f \in \F} f \geq c\right\} = \bigcup_{f \in \F} \{f \geq c\}$, the order-measurability of $f$ and the closure of $\H$ under unions.
            \end{proof}
            
        \subsection{Proof of Lemma \ref{lem:pos_hom}}
            \begin{proof}
                We have
                \begin{align*}
                    \int^E af d\e
                        &= \sup_{c > 0} \frac{c}{\e(\{af \geq c\})}
                        = \sup_{c > 0} \frac{c}{\e(\{f \geq c/a\})} \\
                        &= \sup_{b > 0} \frac{ab}{\e(\{f \geq b\})}
                        = a\sup_{b > 0} \frac{b}{\e(\{f \geq b\})}
                        = a \int^E f d\e.
                \end{align*}
            \end{proof}

        \subsection{Proof of Lemma \ref{lem:indicator}}
            \begin{proof}
                We have $\{1_H \geq c\} = H$ if $0 < c \leq 1$ and $\{1_H \geq c\} = \emptyset$ if $c > 1$.
                Hence,
                \begin{align*}
                    \int^E 1_H d\e
                        = \sup_{c > 0} \frac{c}{\e(\{1_H \geq c\})}
                        = \max\left(\sup_{0 < c \leq 1} \frac{c}{\e(H)},\ \sup_{c > 1} \frac{c}{\e(\emptyset)}\right)
                        = \frac{1}{\e(H)}.
                \end{align*}
            \end{proof}

        \subsection{Proof of Lemma \ref{lem:monotonicity}}
            \begin{proof}
                As $f \geq g$, $\{f \geq c\} \supseteq \{g \geq c\}$.
                As $\e$ is an E-capacity, $\e(\{f \geq c\}) \leq \e(\{g \geq c\})$.
                Dividing $c \geq 0$ by both sides yields $c / \e(\{f \geq c\}) \geq c / \e(\{g \geq c\})$.
                Taking the supremum over $c > 0$ on both sides yields the result $
                    \int^E f d\e
                        \equiv \sup_{c > 0} c / \e(\{f \geq c\})
                        \geq \sup_{c > 0} c / \e(\{g \geq c\})
                        \equiv \int^E g d\e$.
            \end{proof}

        \subsection{Proof of Lemma \ref{lem:sup_interchange}}
            \begin{proof}
                Define the function $g = \sup_{f \in \F} f$.
                The function $g$ is order-measurable since $\{g \geq c\} = \{\sup_{f \in \F} f \geq c\} = \cup_{f \in \F} \{f \geq c\}$ and $\mathcal{H}$ is closed under unions.
                The E-capacity claim follows from applying monotonicity to $g \geq f$ and taking the supremum over $f \in \F$ on both sides.
                For E-measures, we have
                \begin{align*}
                    \int^E \sup_{f \in \F} f d\e
                        &= \sup_{c > 0} \frac{c}{\e(\cup_{f \in \F} \{f \geq c\})}
                        = \sup_{c > 0} \frac{c}{\inf_{f \in \F}\e(\{f \geq c\})} \\
                        &= \sup_{c > 0} \sup_{f \in \F}\frac{c}{\e(\{f \geq c\})}
                        = \sup_{f \in \F} \sup_{c > 0} \frac{c}{\e(\{f \geq c\})}
                        = \sup_{f \in \F} \int^E f d\e.
                \end{align*}
            \end{proof}            

        \subsection{Proof of Lemma \ref{lem:dirac}}    
            \begin{proof}
                We have,
                \begin{align*}
                    \int^E f d(1/\delta_P)
                        = \sup_{c > 0} \frac{c}{(1/\delta_P)(\{f \geq c\})}
                        = \sup_{c > 0} c\ind{f(P) \geq c}
                        = f(P).
                \end{align*}
            \end{proof}

        \subsection{Proof of Lemma \ref{lem:1-E-measure-integral}}
            \begin{proof}
                We have 
                \begin{align*}
                    \frac{c}{\mathbf{1}(\{f \geq c\})}
                        =
                        \begin{cases}
                            c, \quad \textnormal{ if } \{f \geq c\} \neq \emptyset, \\
                            0, \quad \textnormal{ if } \{f \geq c\} = \emptyset.
                        \end{cases}
                \end{align*}
                This means the E-integral becomes 
                \begin{align*}
                    \int^E f d\mathbf{1}
                        &\equiv \sup_{c > 0} \frac{c}{\mathbf{1}(\{f \geq c\})}
                        = \sup\{c > 0 : \{f \geq c\} \neq \emptyset\} \\
                        &= \sup\{c > 0 : \sup_{P \in \P} f(P) \geq c\}
                        = \sup_{P \in \P} f(P).
                \end{align*}
            \end{proof}

        \subsection{Proof of Proposition \ref{prp:least_true_integration}}
            \begin{proof}
                We start with proving the $\F = \{f\}$ case.
                
                As $f$ is $\H$-order-measurable, the hypothesis
                \begin{align*}
                    H_f(P) 
                        := \{P' \in \P : f(P') \geq f(P)\}
                \end{align*}
                is in $\H$.
                Moreover, as $P \in H_f(P)$, Lemma \ref{lem:intersection_closed} implies $H_P \subseteq H_f(P)$.
    
                To start, we show $f(P') = \sup_{P \in \P} f(P) \ind{H_P}(P')$, $P' \in \P$.
                To show this, first observe that for every $P' \in \P$,
                \begin{align*}
                    \sup_{P \in \P} f(P)\ind{H_P}(P') = \sup_{P \in \P: P' \in H_P} f(P).
                \end{align*}
                Now, since $H_P \subseteq H_f(P)$, we have that $P' \in H_P$ implies $f(P') \geq f(P)$.
                Hence,
                \begin{align*}
                    \sup_{P \in \P : P' \in H_P} f(P) \leq f(P')
                \end{align*}
                Conversely, since $P' \in H_{P'}$, we obtain
                \begin{align*}
                     \sup_{P \in \P: P' \in H_P} f(P) \geq f(P').
                \end{align*}
                Combining the two yields $f = \sup_{P \in \P} f(P) \ind{H_P}$.
                
                Now, as $\e$ is an E-measure, we may interchange E-integration and supremum by Lemma \ref{lem:sup_interchange}:
                \begin{align*}
                    \int^E f(P')\,d\e(P')
                        &= \int^E \sup_{P \in \P} f(P) \ind{H_P}(P')\,d\e(P') \\
                        &= \sup_{P \in \P} \int^E f(P) \ind{H_P}(P')\,d\e(P').
                \end{align*}
                Applying Lemma \ref{lem:pos_hom} (positive homogeneity) and Lemma \ref{lem:indicator} (indicators) yields
                \begin{align*}
                    \int^E f(P) \ind{H_P}(P')\,d\e(P')
                        = f(P) \int^E \ind{H_P}(P')\,d\e(P')
                        = \frac{f(P)}{\e(H_P)}.
                \end{align*}
                Taking the supremum over $P \in \P$ proves
                \begin{align}\label{eq:proof_least_true_integration}
                    \int^E f \, d\e = \sup_{P \in \P} \frac{f(P)}{\e(H_P)}.
                \end{align}

                Now, define $g := \sup_{f \in \F} f$, which is $\H$-order-measurable by Lemma \ref{lem:measurable_supremum}.
                Applying \eqref{eq:proof_least_true_integration} with $f = g$ yields 
                \begin{align*}
                     \int^E \sup_{f \in \F} f \, d\e 
                        = \sup_{f \in \F} \sup_{P \in \P} \frac{f(P)}{\e(H_P)}.
                \end{align*}
            \end{proof}
\end{document}